\documentclass[11pt]{article}

\usepackage{amsmath,amsthm,verbatim,amssymb,amsfonts,amscd, graphicx,amsrefs}
\usepackage{graphics}
\usepackage{authblk}
\usepackage{upgreek}
\usepackage{amsrefs,hyperref}
\usepackage{cleveref}
\theoremstyle{plain}
\newtheorem{theorem}{Theorem}[section]

\newtheorem{conjecture}[theorem]{Conjecture}
\newtheorem{corollary}[theorem]{Corollary}
\newtheorem{definition}[theorem]{Definition}

\newtheorem{proposition}[theorem]{Proposition}
\theoremstyle{remark}
\newtheorem{remark}[theorem]{Remark}
\numberwithin{equation}{section}
\newcommand{\diff}{\mathop{}\!\mathrm{d}}
\DeclareMathOperator*{\Hess}{Hess}
\DeclareMathOperator*{\tr}{tr}
\DeclareMathOperator*{\divr}{div}

\title{On the interaction of strain and vorticity for solutions of the Navier--Stokes equation}
\author[1]{Evan Miller}
\affil[1]{University of Maine,
evan.miller1@maine.edu}

\begin{document}

\maketitle

\begin{abstract}
    In this paper, we prove a new identity for divergence free vector fields, showing that
    \begin{equation*}
    \left<-\Delta S,\omega\otimes\omega\right>
    =0,
    \end{equation*}
    where $S_{ij}=\frac{1}{2}\left(\partial_iu_j
    +\partial_ju_i\right)$ is the symmetric part of the velocity gradient, and $\omega=\nabla\times u$ is the vorticity.
    This identity will allow us to understand the interaction of different aspects of the nonlinearity in the Navier--Stokes equation from the strain and vorticity perspective, particularly as they relate to the depletion of the nonlinearity by advection.
    We will prove global regularity for the strain-vorticity interaction model equation, a model equation for studying the impact of the vorticity on the evolution of strain which has the same identity for enstrophy growth as the full Navier--Stokes equation. We will also use this identity to obtain several new regularity criteria for the Navier--Stokes equation, one of which will help to clarify the circumstances in which advection can work to deplete the nonlinearity, preventing finite-time blowup.
\end{abstract}

\tableofcontents

\section{Introduction}

The Navier--Stokes equation, which governs the motion of viscous, incompressible fluids, is one of the most fundamental equations of fluid mechanics, and is given by
\begin{align}
   \partial_t u- \Delta u+(u\cdot \nabla)u
   +\nabla p&=0 \\ 
   \nabla\cdot u&=0,
\end{align}
where $u\in\mathbb{R}^3$ is the velocity, 
and $p$ is the pressure.
The first equation expresses Newton's second law, $F=ma$,
where $\partial_tu+(u\cdot\nabla)u$ is the acceleration in the Lagrangian frame, $-\Delta u$ expresses the viscous forces due to the internal friction of the fluid, and $\nabla p$ expresses the force acting on the fluid due to pressure.
Using the Helmholtz decomposition, it is possible to remove the pressure term entirely by applying a projection onto the space of divergence free vector fields:
\begin{equation}
   \partial_t u- \Delta u
   +P_{df}\left((u\cdot \nabla)u\right)=0.
\end{equation}
In fact, the two main definitions of solutions, that due to Leray \cite{Leray} and that due to Fujita and Kato \cite{KatoFujita},
make no reference to pressure whatsoever.
We will give precise definitions of these solutions in section 2.

Two other crucially important objects in the study of the Navier--Stokes equation are the strain, $S=\nabla_{sym}u$,
and the vorticity, $\omega=\nabla \times u$.
The strain matrix is the symmetric part of $\nabla u$, with
\begin{equation}
    S_{ij}=\frac{1}{2}\left(
    \partial_i u_j +\partial_j u_i\right),
\end{equation}
while the vorticity is a vector representation of the anti-symmetric part of $\nabla u$.
Physically speaking, the vorticity describes the rotation induced by the fluid flow, while the strain describes the deformation due to the fluid flow.
The strain matrix is always trace free due to the incompressibility constraint because
\begin{equation}
    \tr(S)=\nabla\cdot u=0.
\end{equation}

The evolution equation for the vorticity is given by 
\begin{equation}
    \partial_t\omega-\Delta\omega
    +(u\cdot\nabla)\omega-S\omega=0.
\end{equation}
While the vorticity formulation of the Navier--Stokes equation has been studied exhaustively, there has been much less study of the strain formulation. Most of the work on the strain has focused on the vortex stretching term, $S\omega$, in the vorticity equation, which provides a mechanism for enstrophy growth and consequently for finite-time blowup.
The author studied the evolution equation for strain, including the constraint space of strain matrices $L^2_{st}$ 
 in \cite{MillerStrain}.
The evolution equation for the strain is given by
\begin{equation} \label{straineq}
    \partial_t S -\Delta S +(u\cdot \nabla)S
    +S^2+\frac{1}{4}\omega\otimes\omega
    -\frac{1}{4}|\omega|^2 I_3+\Hess(p)=0.
\end{equation}
For a formal definition of the constraint space $L^2_{st}$, see \Cref{StrainDefL2}.
The author and Sawyer further investigated the properties of this constraint space in \cite{MillerSawyer}, giving a Helmholtz-type decomposition of symmetric matrix valued functions. This yields a natural distributional definition of $L^2_{st}$ as the symmetric matrices orthogonal---with respect to the $L^2$ inner product---to all Hessians, multiples of the identity matrix, and matrices which are both trace free and divergence free. This distributional definition is analogous to defining divergence free vector fields as those vector fields orthogonal to all gradients. See Corollary 1.7 and Definition 1.8 in \cite{MillerSawyer} for more details.

Neustupa and Penel proved in \cite{NeuPen1} that enstrophy growth has the identity
\begin{equation} \label{EntGrowth}
    \frac{\diff}{\diff t}
    \|S(\cdot,t)\|_{L^2}^2=
    -2\|S\|_{\dot{H}^1}^2-4\int\det(S),
\end{equation}
and consequently if there is finite-time blowup at $T_{max}<+\infty$, then for all 
$\frac{2}{p}+\frac{3}{q}=2, \frac{3}{2}<q\leq+\infty$,
\begin{equation}
\int_0^{T_{max}}\left\|\lambda_2^+\right\|_{L^q}^p=+\infty,
\end{equation}
where $\lambda_1\leq \lambda_2 \leq \lambda_3$ are the eigenvalues of $S$ and $\lambda_2^+=\max(0,\lambda_2)$.
See also \cites{NeuPen2,MillerStrain} for further discussion.
Using the projection onto the space of strain matrices, the evolution equation for the strain can also be expressed as
\begin{equation} \label{NSstrain}
    \partial_t S-\Delta S+P_{st}\left(
    (u\cdot \nabla)S\right)
    +P_{st}\left(S^2
    +\frac{1}{4}\omega\otimes\omega \right)=0.
\end{equation}
See \cites{MillerStrain,MillerStrainModel} for details.

We also have the following orthogonality for the strain and a certain linear combination of the quadratic terms involving the strain and the vorticity.
\begin{proposition} \label{ThirdOrderIdentity}
For all $S\in L^3_{st}$,
\begin{align}
    \left<S,\omega\otimes\omega\right>
    &=
    -4\int\det(S) \\
    &=-\frac{4}{3}\int\tr\left(S^3\right) \\
    &=-\frac{4}{3}\left<S^2,S\right>,
\end{align}
where $S=\nabla_{sym} u$ and $\omega=\nabla\times u$.
In particular this implies that for all 
$S\in \dot{H}^\frac{1}{2}_{st}$,
\begin{equation}
    \left<\frac{1}{3}S^2
    +\frac{1}{4}\omega\otimes\omega,
    S\right>=0.
\end{equation}
\end{proposition}

This was proven by the author in \cite{MillerStrainModel}. The key piece of the proof is that, integrating by parts and using the divergence free constraint, it follows that
\begin{equation}
    \int\tr\left(\nabla u\right)^3=0,
\end{equation}
and the rest of the proof is elementary linear algebra.
Considering this identity, it makes sense to consider the Navier--Stokes strain equation in the following way:
\begin{equation}
    \partial_t S-\Delta S+\frac{2}{3}P_{st}\left(S^2\right)
    +P_{st}\left(\frac{1}{3}S^2
    +\frac{1}{4}\omega\otimes\omega\right)
    +P_{st}\left((u\cdot\nabla)S\right)=0.
\end{equation}
Observing that because $\nabla\cdot u=0,$ the advection term does not contribute to the $L^2$ growth, so we have
\begin{align}
    \left<\frac{1}{3}S^2+\frac{1}{4}\omega\otimes\omega,
    S\right>&=0 \\
    \left<(u\cdot\nabla)S,S\right>&=0 
\end{align}
If we neglect the two terms that do not contribute to the $L^2$ growth of the solution, we obtain the following model equation, which we will refer to as the strain self-amplification model equation,
\begin{equation} \label{StrainModel}
    \partial_t S -\Delta S +
    \frac{2}{3}P_{st}\left(S^2\right)=0.
\end{equation}
In \cite{MillerStrainModel}, the author proved the local existence of mild solutions and that these solutions have the same identity for enstrophy growth \eqref{EntGrowth} as the Navier--Stokes strain equation, and that furthermore, for all initial data
$S^0\in H^1_{st}$, such that
\begin{equation}
    -\int\det\left(S^0\right)>
    \frac{3}{4}\left\|S^0\right\|_{\dot{H}^1}^2,
\end{equation}
there is finite-time blowup.
This proves that when we isolate the portion of the nonlinearity that depends on the strain, this gives us a model equation that has finite-time blowup, which is consistent with recent evidence in the fluid mechanics literature suggesting the self-amplification of strain, not vortex stretching, is the main feature of the turbulent energy cascade \cite{CarboneBragg}.

In this paper, rather than isolating the effect of the quadratic nonlinearity $S^2$ on the evolution of the strain, we will consider a model equation that isolates the effect of the quadratic nonlinearity $\omega\otimes\omega$. The strain-vorticity interaction model equation will be given by
\begin{equation} \label{StrainVortModel}
 \partial_t S-\Delta S -\frac{1}{2}P_{st}
 \left( \omega\otimes\omega \right)=0.   
\end{equation}
Note that the vorticity is completely determined by the strain in terms a zero order pseudo-differential operator, with
\begin{equation} \label{VortStrainID}
\omega=-2\nabla \times \divr (-\Delta)^{-1}S,
\end{equation}
Crucially, this equation also has the same identity for enstrophy growth \eqref{EntGrowth} as the strain self-amplification model equation and the Navier--Stokes equation.
The justification for this equation is similar to the justification of the strain self-amplification model equation. If we write the Navier--Stokes strain equation as
\begin{equation}
    \partial_t S-\Delta S-\frac{1}{2}P_{st}
    \left(\omega\otimes\omega\right)
    +P_{st}\left(S^2+\frac{3}{4}\omega\otimes\omega\right)
    +P_{st}\left((u\cdot\nabla)S\right)=0,
\end{equation}
we can observe that
\begin{align}
    \left<S^2+\frac{3}{4}\omega\otimes\omega,
    S\right> &=0 \\
    \left<(u\cdot\nabla)S,S\right> &=0,
\end{align}
and so dropping the two terms that do not contribute to the growth of the $L^2$ norm we obtain the strain-vorticity interaction model equation.
We will show that the strain-vorticity interaction model equation has global smooth solutions for all $S^0\in L^2_{st}$.

\begin{theorem} \label{StrainVortGlobalExistIntro}
For all $S^0 \in L^2_{st}$, there is a unique, global mild solution of the strain-vorticity interaction model equation,
$S\in C\left([0,+\infty);L^2_{st}\right) \cap 
C\left((0,+\infty);H^\infty\right)$.
Furthermore, if $S^0\in H^1_{st}$,
then for all $0<t<+\infty$
\begin{equation}
    \frac{1}{2}\|S(\cdot,t)\|_{\dot{H}^1}^2+
    \int_0^t\|-\Delta S(\cdot,\tau)\|_{L^2}^2 \diff\tau
    =
    \frac{1}{2}\left\|S^0\right\|_{\dot{H}^1}^2.
\end{equation}
\end{theorem}

The key element in the proof of 
Theorem \ref{StrainVortGlobalExistIntro} will be a new identity; we will show that $\omega\otimes\omega$ is orthogonal to 
$-\Delta S$ with respect to the $L^2$ inner product. The precise result is as follows.
\begin{theorem} \label{StrainVortOrthogonalityIntro}
Suppose $S\in H^2_{st},$ with $S=\nabla_{sym}u$ and $\omega =\nabla \times u.$ Then
\begin{equation}
    \left<-\Delta S, \omega \otimes \omega\right>=0.
\end{equation}
\end{theorem}

This result shows that the nonlocal interaction of strain and vorticity is not a driving factor behind any finite-time blowup in the strain formulation. This is quite interesting, because it contrasts strongly with the case in the vorticity formulation, where finite-time blowup can only be driven by the nonlocal interaction of the strain and vorticity from the vortex stretching term $S\omega$. 
There is a large literature of model equations studying the the role of vortex stretching by the strain on the vorticity dynamics; for a non-exhaustive list see \cites{DeGregorio,ConstantinLaxMajda,Hou2018,HouEtAl,HouLei,JiaStewartSverak,ElgindiJeong,EGM,HQWW}.
The strain-vorticity interaction model equation is the first model equation to study the reverse direction: the impact of the vorticity on the strain dynamics.

The identity for enstrophy growth from the vorticity formulation is
\begin{equation}
    \frac{\diff}{\diff t}
    \frac{1}{2} \|\omega(\cdot,t)\|_{L^2}^2
    =-\|\omega\|_{\dot{H}^1}^2
    +\left<S,\omega\otimes\omega\right>,
\end{equation}
so it is clear that in the vorticity formulation, blowup can only come from the nonlocal interaction of strain and vorticity.
In particular blowup requires the alignment of the vorticity with the the positive eigenframe of the strain matrix, where by the positive eigenframe we mean the span of the eigenvector or eigenvectors associated to positive eigenvalues.
This suggests blowup may be a more straightforward question in the strain formulation, because it relates primarily to the self-amplification of strain, which can be expressed locally by the cubic nonlinearity $-4\int\det(S)$, and not by the cubic nonlinearity involving a singular integral kernel
$\left<S,\omega\otimes\omega\right>$. The singular integral kernel in this latter expression comes from the fact that 
$S=\nabla_{sym}\nabla\times(-\Delta)^{-1}\omega$, which is a zero order pseudo-differential operator.

Another equation we can consider is the Navier--Stokes strain equation without advection where we drop the term
$P_{st}((u\cdot\nabla)S)$ from the Navier--Stokes strain equation yielding
\begin{equation} \label{NSSWA}
 \partial_t S-\Delta S +P_{st}\left(S^2
    +\frac{1}{4}\omega\otimes\omega \right)=0.   
\end{equation}
This equation, the strain self-amplification model equation, and the strain-vorticity interaction model equation, can all be seen as special cases of a one-parameter family of model equations.
We will define the $\mu$-NS model equation for some $\mu\in\mathbb{R}$ to be given by
\begin{equation} \label{MuNSeqn}
    \partial_t S-\Delta S
    -\frac{1}{2}P_{st}\left(\omega\otimes\omega\right)
    +\mu P_{st}\left(S^2+\frac{3}{4}\omega\otimes\omega
    \right)=0.
\end{equation}
Note that both \eqref{NSSWA} and \eqref{MuNSeqn} are both closed systems of equations because the vorticity is completely determined by the strain---see \eqref{VortStrainID}.

\begin{theorem} \label{MildMuNS}
For all $\mu\in\mathbb{R},$ and for all $S^0\in L^2_{st},$
there exists a unique mild solution
$S\in C\left([0,T_{max});L^2_{st}\right)$ to the $\mu$-NS model equation, and
\begin{equation}
    T_{max}\geq \frac{1728\pi^4}{\left\|S^0\right\|_{L^2}^4}.
\end{equation}
Furthermore, this solution has higher regularity
$S\in C\left((0,T_{max});H^\infty \right)$, 
and for all $0<t<T_{max}$,
\begin{equation}
    \frac{\diff}{\diff t}
    \|S(\cdot,t)\|_{L^2}^2=
    -2\|S\|_{\dot{H}^1}^2-4\int\det(S).
\end{equation}
\end{theorem}

\begin{remark}
    Note that for all $\mu\in\mathbb{R}$, the $\mu$-NS model equation has the same scale invariance. In particular, if $S$ is a solution of the $\mu$-NS model equation, then $S^\lambda$ is also a solution of the $\mu$-NS model equation where
    \begin{equation}
    S^\lambda(x,t)=
    \lambda^2 S(\lambda x, \lambda^2 t).
    \end{equation}
\end{remark}

\begin{remark}
We will note that each of the model equations discussed above are special cases of the $\mu$-NS model equation: 
the strain-vorticity interaction model equation \eqref{StrainVortModel} 
is the case where $\mu=0$,
the strain self-amplification equation \eqref{StrainVortModel} 
is the case were $\mu=\frac{2}{3}$,
and the Navier--Stokes strain equation without advection
\eqref{NSSWA} is the case where $\mu=1$.
This means that we know that there is global regularity in the case where $\mu=0$, and finite-time blowup in the case where $\mu=\frac{2}{3}$,
which suggests that there will also be finite-time blowup in the case where $\mu=1$.
In fact, this suggests that the case where $\mu=1$, the Navier--Stokes strain equation without advection, will be even more singular than the strain self-amplification model equation.

This is, of course, not a proof or even a sketch of a proof. It would seem to be contradictory for there to be global regularity in the cases $\mu=0$ and $\mu=1$, but finite-time blowup in the case where $\mu=\frac{2}{3}$, but there are no straightforward interpolation theorems available given the nonlinear dynamics of the problem.
While it is far from obvious that the degree of singularity of the behaviour of solutions of the $\mu$-NS equation is monotonic in $\mu$, it does definitely appear at a heuristic level that blowup should occur for the Navier--Stokes strain equation without advection in addition to the strain self-amplification model equation, and moreover at a faster rate.
\end{remark}

\begin{center}
 \begin{tabular}{||c |c| c| |} 
 \hline
 Model equation & Parameter $\mu$ 
 & Behaviour of solutions  \\ [0.5ex] 
 \hline\hline
 Strain-vorticity interaction 
 & $\mu=0$ & Global regularity  \\ [.5ex]
 \hline
 Strain self-amplification 
 & $\mu=\frac{2}{3}$ & Finite-time blowup \\ [.5ex]
 \hline
 Navier--Stokes strain without advection
 & $\mu=1$ & ??? \\ [.5ex] 
 \hline
\end{tabular}
\end{center}

\begin{conjecture}
There exists a mild solution of the Navier--Stokes strain equation without advection $S\in C\left([0,T_{max});L^2_{st}\right)$ that blows up in finite-time $T_{max}<+\infty$, meaning that
\begin{equation}
    \lim_{t\to T_{max}}\|S(\cdot,t)\|_{L^2}=+\infty.
\end{equation}
\end{conjecture}

Using the finite-time blowup for the strain self-amplification model equation, the author proved a number of conditional blowup results for the full Navier--Stokes equation in \cite{MillerStrainModel}.
In the same vein, we will prove two new regularity criteria for the Navier--Stokes equation by considering the full Navier--Stokes equation as a perturbation of the strain-vorticity interaction model equation.

\begin{theorem} \label{StrainVortPerturbedRegCritIntro}
Suppose $u\in C\left([0,T_{max});H^3_{df}\right)$ is a mild solution of the Navier--Stokes equation.
Suppose $0\leq \alpha \leq 1$ and 
$p=\frac{2}{1+\alpha}$.
Then for all $0<t<T_{max}$
\begin{equation}
    \|S(\cdot,t)\|_{\dot{H}^1}^2
    \leq
    \left\|S^0 \right\|_{\dot{H}^1}^2
    \exp \left( C_\alpha \int_0^t
    \frac{\left\|P_{st}\left((u\cdot\nabla)S +S^2+\frac{3}{4}\omega\otimes\omega\right)
    (\cdot,\tau)\right\|_{\dot{H}^\alpha}^p}
    {\|S(\cdot,\tau)\|_{\dot{H}^1}^p} 
    \diff\tau \right),
\end{equation}
where $C_\alpha$ depends only on $\alpha$.
In particular, if $T_{max}<+\infty,$ then
\begin{equation}
    \int_0^{T_{max}}
    \frac{\left\|P_{st}\left((u\cdot\nabla)S +S^2+\frac{3}{4}\omega\otimes\omega\right)
    (\cdot,t)\right\|_{\dot{H}^\alpha}^p}
    {\|S(\cdot,t)\|_{\dot{H}^1}^p} \diff t 
    =+\infty.
\end{equation}
\end{theorem}

\begin{theorem} \label{EndpointRegCritIntro}
Suppose $u\in C\left([0,T_{max});H^3_{df}\right)$ is a mild solution of the Navier--Stokes equation.
Then if $T_{max}<+\infty$, then
\begin{equation}
    \limsup_{t\to T_{max}}
    \frac{\left\|P_{st}\left((u\cdot\nabla)S
    +S^2+\frac{3}{4}\omega\otimes\omega\right)
    (\cdot,t)\right\|_{L^2}}
    {\|-\Delta S(\cdot,t)\|_{L^2}} \geq 1.
\end{equation}
\end{theorem}

Contrasting with the regularity criteria involving
$P_{st}\left((u\cdot\nabla)S+S^2
+\frac{3}{4}\omega\otimes\omega\right)$
that can be obtained perturbatively from the global regularity of solutions of the strain-vorticity interaction model equation, the author previously derived blowup conditions for the full Navier--Stokes equation that can be obtained perturbatively from the finite-time blowup for the strain self-amplification model equation \cites{MillerStrainModel}. 
This yields blowup conditions involving the size of
$P_{st}\left((u\cdot\nabla)S+\frac{1}{3}S^2
+\frac{1}{4}\omega\otimes\omega\right)$,
whose precise statement is as follows.

\begin{theorem} \label{BlowupThmIntro}
Suppose $u\in C\left([0,T_{max});H^2_{df}\right)$ is a mild solution of the Navier--Stokes equation such that
\begin{equation}
    f_0=-3\left\|S^0\right\|_{\dot{H}^1}^2
    -4\int\det\left(S^0\right)>0,
\end{equation}
and for all $0\leq t<T_{max}$
\begin{equation}
    \frac{\left\|P_{st}\left((u\cdot\nabla)S+\frac{1}{3}S^2
    +\frac{1}{4}\omega\otimes\omega
    \right)(\cdot,t)\right\|_{L^2}}
    {\left\|\left(-\Delta u+P_{st}\left(
    \frac{1}{2}(u\cdot\nabla)S+\frac{5}{6}S^2
    +\frac{1}{8}\omega\otimes\omega\right)
    \right)(\cdot,t)\right\|_{L^2}}
    \leq
    2.
\end{equation}
Then there is finite-time blowup with
\begin{equation}
    T_{max}
    <
    T_*
    =
    \frac{-E_0+\sqrt{E_0^2+f_0 K_0}}
    {f_0},
\end{equation}
where
$K_0 = \frac{1}{2}\left\|u^0\right\|_{L^2}^2,
E_0 = \frac{1}{2}\left\|\nabla u^0\right\|_{L^2}^2,$
and $f_0$ is as defined above.
\end{theorem}

\begin{remark}
These perturbative conditions give new insight into the role of advection in the Navier--Stokes regularity problem by providing quantitative estimates involving the advection of strain pointing to either global regularity or finite-time blowup. There is a large body of evidence suggesting that the advection may play a regularizing role in the Navier--Stokes equation, but the results above provide a quantitative estimates for understanding this possible mechanism for the depletion of nonlinearity. 

In the strain formulation, the decisive factor is the alignment of $P_{st}((u\cdot\nabla)S)$ and
$P_{st}\left(S^2+\frac{3}{4}\omega\otimes\omega\right)$.
If the alignment is such that 
$P_{st}\left((u\cdot\nabla)S+\frac{1}{3}S^2
+\frac{1}{4}\omega\otimes\omega\right)$
is small then there will be finite-time blowup, as the solutions of the Navier--Stokes equation in that case are sufficiently close to solutions of the strain self-amplification model equation, for which there is finite-time blowup.
If the alignment means that
$P_{st}\left((u\cdot\nabla)S+S^2
+\frac{3}{4}\omega\otimes\omega\right)$
is small, then there will be global regularity,
as the solutions of the Navier--Stokes equation in that case are sufficiently close to solutions of the strain-vorticity interaction model equation, for which there is global regularity.
Of course, there is also the possibility that neither of these terms are small, in which case our two model equations will not tell us anything about the dynamics of the full Navier--Stokes equation.

The conditions in \Cref{StrainVortPerturbedRegCritIntro,EndpointRegCritIntro,BlowupThmIntro}
could be studied numerically using candidate blowup scenarios. The conditions in \Cref{StrainVortPerturbedRegCritIntro,EndpointRegCritIntro} point to an interaction of advection with the quadratic nonlinearity that depletes the nonlinearity and therefore leads to global regularity, while the condition in \Cref{BlowupThmIntro} points to an interaction of the advection with the quadratic nonlinearity that maintains the growth of enstrophy, leading to finite-time blowup.
While the statements of these perturbative conditions are straightforward, the projections in question are only simple to compute in Fourier space; in physical space they are matrices of Riesz transforms, involving complicated singular integral operators.
\end{remark}

Finally, we will prove a new regularity criterion for solutions of the Navier--Stokes equation when the strain is sufficiently close to being an eigenfunction of the Laplacian.

\begin{theorem} \label{StrainAlmostEigenIntro}
Suppose $u\in C\left([0,T_{max});\dot{H}^1_{df}\right)$
is a mild solution to the Navier--Stokes equation,
and suppose $\frac{2}{p}+\frac{3}{q}=2,
\frac{3}{2}<q\leq +\infty$.
Then for all $0<t<T_{max}$
\begin{equation} 
    \|\omega(\cdot,t)\|_{L^2}^2
    <\left\|\omega^0\right\|_{L^2}^2
    \exp\left(C_q
    \int_0^{t}\inf_{\rho\in\mathbb{R}}
    \|-\rho\Delta S-S\|_{L^q}^p \diff\tau \right),
\end{equation}
where $C_q>0$ depends only on $q$.
In particular, if $T_{max}<+\infty$, then
\begin{equation}
    \int_0^{T_{max}}\inf_{\rho\in\mathbb{R}}
    \|-\rho\Delta S-S\|_{L^q}^p \diff t
    =+\infty.
\end{equation}
\end{theorem}

This extends results for solutions of the Navier--Stokes equation where the velocity is sufficiently close to being an eigenfunction of the Laplacian proven by the author in \cites{MillerEigen}, where the author proved that if a mild solution of the Navier--Stokes equation $u\in C\left(\left[0,T_{max}\right);H^1\right)$
blows up in finite-time $T_{max}<+\infty$,
then for all
$\frac{6}{5}<q\leq 3, \frac{2}{p}+\frac{3}{q}=3$, 
\begin{equation}
    \int_0^{T_{max}} \inf_{\lambda\in\mathbb{R}}
    \|-\Delta u-\lambda u\|_{L^q}^p \diff t
    =+\infty.
\end{equation}

Theorem \ref{StrainAlmostEigenIntro} is an advance over this result for a number of reasons. First of all, it holds for a wider range of exponents: including up until the endpoint case $p=+\infty$, although the endpoint case is not included, whereas the regularity criterion in \cite{MillerEigen} only includes the cases $1\leq p\leq 4$,
and so doesn't get arbitrarily close to the endpoint case $p=+\infty$. 
Theorem \ref{StrainAlmostEigenIntro} also has the advantage of only requiring a solution in mild solution in $\dot{H}^1_{df}$, not $H^1_{df}$, so there is no need to require finite-energy in this case, and of requiring less regularity on portion of the term that does not have a parameter for minimization.

We will note that while Theorem \ref{StrainAlmostEigenIntro} holds for a broader range of exponents than the regularity criteria in \cites{MillerEigen}, neither result implies the other. We can see this by comparing the $q=2$ case of both results, in which case the infimum can be computed explicitly.
In these cases we find that 
if $T_{max}<+\infty$, then
\begin{align}
    \int_0^{T_{max}}\left(
    1-\frac{\|S\|_{\dot{H}^1}^4}{\|S\|_{L^2}^2
    \|-\Delta S\|_{L^2}^2}\right)^2
    \|S\|_{L^2}^4 \diff t
    &= +\infty \\
    \int_0^{T_{max}}\left(
    1-\frac{\|\nabla u\|_{L^2}^4}
    {\|u\|_{L^2}^2\|-\Delta u\|_{L^2}^2}
    \right)^\frac{2}{3} 
    \|-\Delta u\|_{L^2}^\frac{4}{3}
    \diff t
    &= +\infty.
\end{align}
The results are different even though there are structural similarities.

We will also note that in the endpoint case, 
$L^\infty_t L^\frac{3}{2}_x$, we cannot guarantee the blowup of the norm of the infimum if $T_{max}<+\infty$,
but we do have a lower bound.
\begin{theorem} \label{StrainAlmostEigenEndpointIntro}
Suppose $u\in C\left([0,T_{max});\dot{H}^1_{df}\right)$
is a mild solution to the Navier--Stokes equation
that blows up in finite-time $T_{max}<+\infty$.
Then
\begin{equation}
    \limsup_{t\to T_{max}} 
    \inf_{\rho\in\mathbb{R}}
    \|-\rho\Delta S-S\|_{L^\frac{3}{2}}
    \geq
    \frac{1}{3}
    \left(\frac{2}{\pi}\right)^\frac{4}{3}.
\end{equation}
\end{theorem}

In section \ref{Definitions}, we will give a number of important definitions, as well as collect a few simple propositions and classical theorems from earlier works that will be useful to us.
In section \ref{MainID}, we will prove the main new identity, Theorem \ref{StrainVortOrthogonalityIntro}.
In section \ref{MildSolutions}, we will deal with the local wellposedness theory for solutions of the $\mu$-NS model equation, proving Theorem \ref{MildMuNS}.
In section \ref{InteractionModel}, we will establish global regularity for solution of the strain-vorticity interaction model equation, proving Theorem \ref{StrainVortGlobalExistIntro}, and will also prove the regularity criteria in Theorems \ref{StrainVortPerturbedRegCritIntro} and \ref{EndpointRegCritIntro}.
Finally, in section \ref{RegCritEigen}, we will consider the regularity criterion for solutions of the Navier--Stokes equation where the strain is sufficiently close to being an eigenfunction of the Laplacian, proving Theorems \ref{StrainAlmostEigenIntro} and \ref{StrainAlmostEigenEndpointIntro}.

\section{Definitions and preliminaries}  \label{Definitions}

In this section, we will give a number of definitions of our spaces as well as definitions of mild solutions. In addition, we will collect a number of standard results and some useful propositions from the the author's previous work.
We will begin by defining the spaces of divergence free vector fields and strain matrices.
The space of divergence free vector fields can be defined very straightforwardly on the Fourier space side.

\begin{definition}
The space of divergence free vector fields in $L^2$
is given by
\begin{equation}
    L^2_{df}=\left\{
    u\in L^2\left(\mathbb{R}^3;\mathbb{R}^3\right)
    : \xi\cdot \hat{u}(\xi)=0, \text{almost everywhere}, 
    \, 
    \xi\in\mathbb{R}^3\right\}.
\end{equation}
\end{definition}

The space of strain matrices can then be defined in terms of the space of divergence free vector fields.
First we will define the symmetric gradient operator.

\begin{definition}
Suppose $v:\mathbb{R}^3\to \mathbb{R}^3$, then the symmetric gradient of $v$ is given by
\begin{equation}
    \left(\nabla_{sym}v\right)_{ij}
    =
    \frac{1}{2}\left(\partial_i v_j+\partial_j v_i\right).
\end{equation}
\end{definition}

If $v\in C^1\left(\mathbb{R}^3;\mathbb{R}^3\right)$,
then this is a derivative in the classical sense, and can otherwise be taken as a derivative in the distributional sense.
Now we will define the space of strain matrices.

\begin{definition} \label{StrainDefL2}
The space of strain matrices in $L^2$ is given by
\begin{equation}
    L^2_{st}=\left\{\nabla_{sym}u: u\in\dot{H}^1_{df}\right\}.
\end{equation}
\end{definition}
We will note that the spaces 
$H^\alpha_{df}, \dot{H}^\alpha_{df}, 
H^\alpha_{st},$ and $\dot{H}^\alpha_{st}$
are defined entirely analogously to the definitions above,
so we will not clog up the paper by giving separate definitions for each of these spaces.
In order to define the space $L^q_{st}$, the strain constraint space in $L^q$ we will need to make use of the Riesz transform, as the Fourier side characterization isn't available when $q>2$.

\begin{definition} \label{StrainDefLq}
For all $1<q<+\infty$, we define the strain constraint space $L^q_{st}$ by
\begin{equation}
    L^q_{st}=\left\{
    S\in L^q\left(\mathbb{R}^3:\mathbb{S}^{3\times 3}\right):
    \tr(S)=0, S+2\nabla_{sym}\divr(-\Delta)^{-1}S=0
    \right\}.
\end{equation}
\end{definition}

We should note that in the case where $q=2$, \Cref{StrainDefL2,StrainDefLq} give two different definitions of the space $L^2_{st}$.
The author showed the equivalency of these definitions in \cite{MillerStrain}.

Another result that will be important in our analysis is a Hilbert space isometry relating the strain and the vorticity.
This isometry is classical in the fluid mechanics literature, but for details see \cite{MillerStrain}.
\begin{proposition}
For all $-\frac{3}{2}<\alpha<\frac{3}{2}$,
and for all $S\in \dot{H}^\alpha_{st}$,
\begin{equation}
    \|S\|_{\dot{H}^\alpha}^2
    =
    \frac{1}{2}\|\omega\|_{\dot{H}^\alpha}^2
    =
    \frac{1}{2}\|\nabla u\|_{\dot{H}^\alpha}^2,
\end{equation}
where $S=\nabla_{sym} u$ and $\omega=\nabla\times u$.
\end{proposition}

The Sobolev inequality was first proven by Sobolev \cite{Sobolev} in the case where $s=1$.
The sharp version of this inequality was proven by Talenti \cite{Talenti} in the case where $s=1$, and the general sharp version of this inequality with $0<s<\frac{3}{2}$ was proven by Lieb \cite{Lieb}, and is stated below.

\begin{theorem} \label{Sobolev}
    Suppose $0<s<\frac{3}{2},$ and 
    $\frac{1}{q}=\frac{1}{2}-\frac{s}{3}.$
    Then for all 
    $f\in\dot{H}^s\left(\mathbb{R}^3\right),$
    \begin{equation}
        \|f\|_{L^q}\leq C_s \|f\|_{\dot{H}^s},
    \end{equation}
    where
    \begin{equation}
        C_s=2^{-\frac{s}{3}}\pi^{-\frac{2}{3}s}
        \left(\frac{\Gamma\left(\frac{3}{2}-s\right)}
        {\Gamma\left(\frac{3}{2}+s\right)}
        \right)^\frac{1}{2}
    \end{equation}
    In particular, the sharp constant in the case where $s=1$
    will be important in the proof of Theorem \ref{StrainAlmostEigenEndpointIntro}.
    In this case, the Sobolev inequality states that for all 
    $f\in\dot{H}^1\left(\mathbb{R}^3\right)$
\begin{equation}
    \|f\|_{L^6}\leq \frac{1}{\sqrt{3}} 
    \left(\frac{2}{\pi}\right)^\frac{2}{3}
    \|f\|_{\dot{H}^1}.
\end{equation}
    Note that the scaling relation between the parameters $q$ and $s$ can be stated equivalently as 
    \begin{equation}
        s=\frac{3}{2}-\frac{3}{q}.
    \end{equation}
\end{theorem}

Now we will give the definition of mild solutions to the Navier--Stokes equation and the $\mu$-NS model equation.

\begin{definition} \label{KatoMild}
$u\in C\left([0,T_{max}); \dot{H}^1_{df}\right)$ is a mild solution of the Navier--Stokes equation 
with initial data $u^0\in \dot{H}^1_{df}$, 
if for all $0<t<T_{max}$
\begin{equation}
    u(\cdot,t)= e^{t\Delta}u^0
    -\int_0^t e^{(t-\tau)\Delta} 
    P_{df}((u\cdot\nabla)u)(\cdot,\tau) \diff\tau.
\end{equation}
\end{definition}

\begin{definition}
$S\in C\left([0,T_{max}); L^2_{st}\right)$ is a mild solution of the $\mu$-NS model equation with parameter 
$\mu\in\mathbb{R}$
with initial data $S^0\in L^2_{st}$, 
if for all $0<t<T_{max}$
\begin{equation}
    S(\cdot,t)= e^{t\Delta}S^0
    -\int_0^t e^{(t-\tau)\Delta} 
    P_{st}\left(\mu\left(S^2
    +\frac{3}{4}\omega\otimes\omega\right)
    -\frac{1}{2}\omega\otimes\omega
    \right)(\cdot,\tau) \diff\tau.
\end{equation}
\end{definition}

\begin{remark}
We will note that this definition is well defined 
because the map
\begin{equation}
    S \longmapsto 
    e^{t\Delta}S^0
    -\int_0^t e^{(t-\tau)\Delta} 
    P_{st}\left(\mu\left(S^2
    +\frac{3}{4}\omega\otimes\omega\right)
    -\frac{1}{2}\omega\otimes\omega
    \right)(\cdot,\tau) \diff\tau,
\end{equation}
is a well defined map from
$C\left([0,T_{max});L^2_{st}\right)$ to itself.
Definition \ref{KatoMild} was introduced by Fujita and Kato in \cite{KatoFujita}, where they proved the local-in-time existence of mild solutions based on a fixed point argument of the above map. Our proof of the local-in-time existence of mild solutions will be adapted from their proof.
\end{remark}

\begin{remark}
Now that we have defined mild solutions of the $\mu$-NS model equation, we have also defined mild solutions of the strain-vorticity interaction model equation, the strain self-amplification model equation, and the Navier--Stokes strain equation without advection, as these are simply particular cases of the $\mu$-NS model equation, with the parameter $\mu=0, \mu=\frac{2}{3}$, and $\mu=1$ respectively.
\end{remark}

\section{Main identity} \label{MainID}

In this section, we will prove the main new identity involving strain and vorticity that is key to all the arguments in this paper.
We will now prove \Cref{StrainVortOrthogonalityIntro}, which is restated for the reader's convenience.

\begin{theorem} \label{StrainVortOrthogonality}
Suppose $S\in H^2_{st},$ with $S=\nabla_{sym}u$ and $\omega =\nabla \times u.$ Then
\begin{equation}
    \left<-\Delta S, \omega \otimes \omega\right>=0.
\end{equation}
\end{theorem}

\begin{proof}
First we will note that $S\in H^2_{st}$ implies that $-\Delta S\in L^2.$ We can also note that $\omega \in H^2_{df},$ and therefore, $\omega \in L^4$ and consequently $\omega\otimes\omega \in L^2.$ This implies the integral is absolutely integrable, and we have enough regularity to integrate by parts, using the fact that $\omega \otimes \omega$ is symmetric and $\nabla \cdot \omega=0$ to conclude
\begin{align}
    \left<-\Delta S, \omega \otimes \omega\right>&=
    \left<\nabla (-\Delta u), \omega \otimes \omega\right>
    \\ &=
    -\left<-\Delta u, 
    \divr\left(\omega \otimes \omega\right)\right>\\
    &=
    -\left<-\Delta u, 
    (\omega \cdot \nabla)\omega\right>.
\end{align}
Using the divergence free condition $\nabla \cdot \omega=0,$ we can conclude that
\begin{align}
    (\omega \cdot \nabla)\omega&=
    (\nabla \times \omega) \times \omega
    +\nabla\frac{1}{2}|\omega|^2\\
    &=
    -\Delta u \times \omega
    +\nabla\frac{1}{2}|\omega|^2\\
\end{align}
Applying this identity we find that
\begin{align}
    \left<-\Delta S, \omega \otimes \omega\right>&=
    -\left<-\Delta u, -\Delta u \times \omega \right>
    -\left<-\Delta u, \nabla \frac{1}{2}|\omega|^2\right>\\
    &=0.
\end{align}
We will note here that $-\Delta u \times \omega$ is orthogonal to $-\Delta u$ point-wise in physical space, while $\nabla \frac{1}{2}|\omega|^2$ is orthogonal to $-\Delta u$ in Fourier space due to the Helmholtz decomposition because $\nabla \cdot (-\Delta u)=0.$ This completes the proof.
\end{proof}

\section{Mild solutions of the $\mu$-NS model equation} \label{MildSolutions}

In this section, we will prove Theorem \ref{MildMuNS}, breaking the different parts of this result into pieces.
First, we will construct mild solutions for short times.

\begin{theorem} \label{MildExistence}
For all $\mu\in\mathbb{R},$ and for all $S^0\in L^2_{st},$
there exists a unique mild solution
$S\in C\left([0,T_{max});L^2_{st}\right)$ to the $\mu$-NS model equation, and
\begin{equation}
    T_{max}\geq \frac{C_\mu}{\left \|S^0\right\|_{L^2}^4},
\end{equation}
where 
\begin{equation}
    C_\mu
    =
    \frac{1}{8^5\left(2|\mu|
    +\left|3\mu-2\right|\right)^4}
\end{equation}
Furthermore, this solution has higher regularity
$S\in C\left((0,T_{max});H^\infty \right)$.
\end{theorem}

\begin{proof}
    The approach of this proof will follow the classic methods of Fujita and Kato \cite{KatoFujita} for constructing mild solutions in the subcritical case.
    We begin by fixing 
\begin{equation}
    T<\frac{C_\mu}
    {\left\|S^0\right\|_{L^2}^4}.
\end{equation}
We define the map $Q: C\left([0,T]:L^2_{st}\right)
\to C\left([0,T]:L^2_{st}\right)$
by
\begin{equation}
Q[S](\cdot,t)= e^{t\Delta} S^0
-\int_0^t P_{st} e^{(t-\tau)\Delta}
\left(\mu S^2+ 
\left(\frac{3}{4}\mu-\frac{1}{2}\right)
\omega\otimes\omega \right)(\cdot,\tau)
\diff\tau.
\end{equation}
Note that $P_{st}$ commutes with the heat kernel, and so $S$ is a mild solution of the $\mu$-NS model equation if and only if $S$ is a fixed point of $Q$ with $Q[S]=S$.

Suppose that
\begin{equation}
    \|S\|_{C_T L^2_x}
    \leq 
    2\left\|S^0\right\|_{L^2}.
\end{equation}
Then for all $0<t\leq T$,
we can compute that
\begin{equation}
    \|Q[S](\cdot,t)\|_{L^2}
    \leq 
    \|e^{t\Delta} S^0\|_{L^2}
    +\int_0^t \left\|
     e^{(t-\tau)\Delta}
    \left(\mu S^2+ 
    \left(\frac{3}{4}\mu-\frac{1}{2}\right)
\omega\otimes\omega \right)(\cdot,\tau)
    \right\|_{L^2} \diff\tau.
\end{equation}
Recall that the heat semigroup is defined by
\begin{equation}
    e^{t\Delta}f=
    G(\cdot,t)*f,
\end{equation}
where
\begin{align}
    G(x,t)
    &=
    \frac{1}{t^\frac{3}{2}}
    g\left(\frac{x}{t^\frac{1}{2}}
    \right) \\
    g(x)
    &=
    \frac{1}{(4\pi)^\frac{3}{2}}
    \exp\left(-\frac{|x|^2}{4}\right).
\end{align}
Applying Young's convolution inequality, we can find that
\begin{align}
    \|e^{t\Delta}f\|_{L^2}
    &\leq 
    \|G(\cdot,t)\|_{L^1} \|f\|_{L^2} \\
    &=
    \|f\|_{L^2},
\end{align}
and that
\begin{align}
    \|e^{t\Delta}f\|_{L^2}
    &\leq 
    \|G(\cdot,t)\|_{L^2} \|f\|_{L^1} \\
    &=
    \frac{1}{t^\frac{3}{4}} 
    \|g\|_{L^2}\|f\|_{L^1}.
\end{align}
Applying these estimates, we can compute that
\begin{equation}
    \left\|e^{t\Delta} S^0\right\|_{L^2}
    \leq 
    \left\|S^0\right\|_{L^2},
\end{equation}
and that
\begin{align} 
    \left\|
     e^{(t-\tau)\Delta}
    \left(\mu S^2+ 
    \left(\frac{3}{4}\mu-\frac{1}{2}\right)
\omega\otimes\omega \right)(\cdot,\tau)
    \right\|_{L^2}
    &\leq 
    \frac{\|g\|_{L^2}}{(t-\tau)^\frac{3}{4}}
    \left\|\left(\mu S^2+ 
    \left(\frac{3}{4}\mu-\frac{1}{2}\right)
    \omega\otimes\omega 
    \right)(\cdot,\tau)\right\|_{L^1} \\
    &\leq 
    \frac{\|g\|_{L^2}}{(t-\tau)^\frac{3}{4}}
    |\mu|\|S(\cdot,\tau)\|_{L^2}^2
    +\left|\frac{3}{4}\mu-\frac{1}{2}\right|
    \|\omega(\cdot,\tau)\|_{L^2}^2 \\
    &=
    \frac{\|g\|_{L^2}}{(t-\tau)^\frac{3}{4}}
    \left(|\mu|
    +\left|\frac{3}{2}\mu-1\right|\right)
    \|S(\cdot,\tau)\|_{L^2}^2 \\
    &\leq 
    \frac{\|g\|_{L^2}}{(t-\tau)^\frac{3}{4}}
    \left(4|\mu|
    +\left|6\mu-4\right|\right)
    \left\|S^0\right\|_{L^2}^2.
\end{align}
It is a simple computation that
\begin{equation}
    \int_0^t\frac{1}{(t-\tau)^\frac{3}{4}}
    =4t^\frac{1}{4},
\end{equation}
and so we can observe that
\begin{equation}
    \int_0^t \left\|
     e^{(t-\tau)\Delta}
    \left(\mu S^2+ 
    \left(\frac{3}{4}\mu-\frac{1}{2}\right)
\omega\otimes\omega \right)(\cdot,\tau)
    \right\|_{L^2} \diff\tau
    \leq   
    8\left(2|\mu|
    +\left|3\mu-2\right|\right)
    \|g\|_{L^2}
    \left\|S^0\right\|_{L^2}^2
    T^\frac{1}{4}.
\end{equation}
Compute that
\begin{equation}
    \|g\|_{L^2}
    =
    \frac{1}{8^\frac{1}{4}},
\end{equation}
and so recalling that by hypothesis
\begin{equation}
T<\frac{1}
    {\left(8\left(2|\mu|
    +\left|3\mu-2\right|\right)\right)^4
    \|g\|_{L^2}^4
    \left\|S^0\right\|_{L^2}^4},
\end{equation}
we can conclude
\begin{equation}
    8\left(2|\mu|
    +\left|3\mu-2\right|\right)
    \|g\|_{L^2}
    \left\|S^0\right\|_{L^2}
    T^\frac{1}{4}
    <1.
\end{equation}
This implies that
\begin{equation}
    \|Q[S]\|_{C_T L^2_x}
    \leq 
    2\left\|S^0\right\|_{L^2}.
\end{equation}

Define the closed ball $B_{2\left\|S^0\right\|_{L^2}}\subset 
C\left([0,T];L^2_{st}\right)$ by
\begin{equation}
    B_{2\left\|S^0\right\|_{L^2}}
    =\left\{
    S\in C\left([0,T];L^2_{st}\right):
    \|S\|_{C_T L^2_x}\leq 
    2\left\|S^0\right\|_{L^2}
    \right\}.
\end{equation}
We have just proven the $Q$ is an automorphism on $B_{2\left\|S^0\right\|_{L^2}}$.
Now observe that for all 
$S,\Tilde{S}\in
B_{2\left\|S^0\right\|_{L^2}}$,
we have
\begin{multline}
    Q[S](\cdot,t)-Q[\Tilde{S}](\cdot,t)
    =
    \int_0^t P_{st} e^{(t-\tau)\Delta}
    \Bigg(\left(\mu \Tilde{S}^2+ 
\left(\frac{3}{4}\mu-\frac{1}{2}\right)
\Tilde{\omega}\otimes\Tilde{\omega} \right) \\
-\left(\mu S^2+ 
\left(\frac{3}{4}\mu-\frac{1}{2}\right)
\omega\otimes\omega \right)\Bigg)(\cdot,\tau)
\diff\tau,
\end{multline}
and therefore
\begin{multline}
Q[S](\cdot,t)-Q[\Tilde{S}](\cdot,t)
=
\frac{1}{2}
\int_0^t P_{st} e^{(t-\tau)\Delta}
    \Bigg(\mu ((\Tilde{S}+S)(\Tilde{S}-S)
    +(\Tilde{S}-S)(\Tilde{S}+S)) \\
+\left(\frac{3}{4}\mu-\frac{1}{2}\right)
\left((\omega+\Tilde{\omega})\otimes 
(\omega-\Tilde{\omega})
+(\omega-\Tilde{\omega})\otimes 
(\omega+\Tilde{\omega})
\right)\Bigg)(\cdot,\tau)
\diff\tau.
\end{multline}
Therefore, we may compute, using the estimates for the heat kernel as above, that
\begin{multline}
\|Q[S](\cdot,t)-Q[\Tilde{S}](\cdot,t)\|_{L^2}
\leq 
\int_0^t
\frac{1}{(t-\tau)^\frac{3}{4}}
\|g\|_{L^2}
\Bigg(|\mu|\|S+\Tilde{S}\|_{C_T L^2_x}
\|S-\Tilde{S}\|_{C_T L^2_x} \\
+\left|\frac{3}{4}\mu-\frac{1}{2}\right|
\|\omega+\Tilde{\omega}\|_{C_T L^2_x}
\|\omega-\Tilde{\omega}\|_{C_T L^2_x}
\Bigg) \diff \tau,
\end{multline}
and that consequently,
for all $0<t \leq T$,
\begin{equation}
    \|Q[S](\cdot,t)-Q[\Tilde{S}](\cdot,t)\|_{L^2}
    \leq 
    8\left(2|\mu|
    +\left|3\mu-2\right|\right)
    \|g\|_{L^2}
    \left\|S^0\right\|_{L^2}
    T^\frac{1}{4}
    \|S-\Tilde{S}\|_{C_T L^2_x}.
\end{equation}
Let 
\begin{align}
    r&:= 
    8\left(2|\mu|
    +\left|3\mu-2\right|\right)
    \|g\|_{L^2}
    \left\|S^0\right\|_{L^2}
    T^\frac{1}{4} \\
    &<1,
\end{align}
and we can see that for all 
$S,\Tilde{S}\in
B_{2\left\|S^0\right\|_{L^2}}$,
\begin{equation}
    \left\|Q[S]-Q[\Tilde{S}]\right\|_{C_T L^2_x}
    \leq r
    \left\|S-\Tilde{S}
    \right\|_{C_T L^2_x}.
\end{equation}
Applying the Banach fixed point Theorem, we can see that there exists a unique fixed point $S^*\in B_{2\left\|S^0\right\|_{L^2}}$ such that 
\begin{equation}
    Q[S^*]=S^*.
\end{equation}
We have now shown that there exists a unique mild solution locally in time.

We can bootstrap higher regularity for all positive times up until the blowup time, by making use of the smoothing due to the heat kernel. The idea is to put a portion of the derivative on the heat kernel to get a little more regularity each step, and then use induction to conclude that the solution is smooth. The method is classical, so we will not get into the details here. See \cite{KatoFujita} for the details of the method for the Navier--Stokes equation.
\end{proof}

We have now constructed mild solutions locally in time, with the time of existence uniform in the $L^2$ norm; however, unlike in the proof of \Cref{MildMuNS}, the time of existence is not uniform in $\mu$. In order to prove the time of existence is uniform in $\mu$, we will need to prove that the identity for enstrophy growth for the Navier--Stokes equation also holds for the $\mu$-NS model equation.

\begin{proposition} \label{EnstrophyProp}
    Suppose $S\in C\left([0,T_{max});L^2_{st}\right)$ is a mild solution to the $\mu$-NS model equation 
    for some $\mu\in\mathbb{R}$,
    then for all $0<t<T_{max}$,
    \begin{equation}
    \frac{\diff}{\diff t}
    \|S(\cdot,t)\|_{L^2}^2=
    -2\|S\|_{\dot{H}^1}^2 -4\int\det(S).
\end{equation}
\end{proposition}

\begin{proof}
    Applying \Cref{ThirdOrderIdentity},
    we find that for all $0<t<T_{max}$,
    \begin{align}
    \frac{\diff}{\diff t}
    \|S(\cdot,t)\|_{L^2}^2
    &=
    -2\|S\|_{\dot{H}^1}^2
    -\mu \left<
    S,S^2+\frac{3}{4}
    \omega\otimes\omega\right>
    +\left<S,\omega\otimes\omega\right> \\
    &=
    -2\|S\|_{\dot{H}^1}^2
    -4\int\det(S),
    \end{align}
    which completes the proof.
\end{proof}

\begin{proposition} \label{EnstrophyBoundProp}
    Suppose $S\in C\left([0,T_{max});L^2_{st}\right)$ is a mild solution to the $\mu$-NS model equation 
    for some $\mu\in\mathbb{R}$,
    then for all $0<t<T_{max}$,
    \begin{equation} \label{DiffBound}
    \frac{\diff}{\diff t}
    \|S(\cdot,t)\|_{L^2}^2=
    \frac{1}{3456\pi^4} \|S(\cdot,t)\|_{L^2}^6.
\end{equation}
Furthermore, this differential inequality implies that
for all $0<t<T_{max}$,
\begin{equation}
    \|S(\cdot,t)\|_{L^2}^2\leq
    \frac{\left\|S^0\right\|_{L^2}^2}
    {\sqrt{1-\frac{1}{1728\pi^4}\left\|S^0\right\|_{L^2}^4 t}}.
\end{equation}
\end{proposition}

\begin{proof}
    The author proved in \cite{MillerAlmost2D} that for all $S\in H^1_{st}$,
    \begin{equation}
    -2\|S\|_{\dot{H}^1}^2
    -4\int\det(S)
    \leq 
    \frac{1}{3456\pi^4} \|S(\cdot,t)\|_{L^2}^6,
    \end{equation}
    which gives the bound \eqref{DiffBound}.
    Integrating this differential inequality completes the proof.
\end{proof}

\begin{corollary} \label{TimeBoundCor}
For all $\mu\in\mathbb{R},$ and for all $S^0\in L^2_{st},$
there exists a unique mild solution
$S\in C\left([0,T_{max});L^2_{st}\right)$ to the $\mu$-NS model equation, and
\begin{equation}
    T_{max}\geq \frac{1728\pi^4}{\left\|S^0\right\|_{L^2}^4}.
\end{equation}
\end{corollary}

\begin{proof}
    We have already shown the existence of a unique mild solution locally in time in \Cref{MildExistence}, and furthermore it is clear that if $T_{max}<+\infty$,
    then
    \begin{equation}
    \lim_{t \to T_{max}}
    \|S(\cdot,t)\|_{L^2}^2
    =+\infty.
    \end{equation}
    Applying \Cref{EnstrophyBoundProp}, we can clearly see that
    \begin{equation}
    T_{max}\geq \frac{1728\pi^4}{\left\|S^0\right\|_{L^2}^4},
    \end{equation}
    and this completes the proof.
\end{proof}

\begin{remark}
    We have now proven \Cref{MildMuNS}, including the lower bound on the blowup time and the identity for enstrophy growth.
\end{remark}

\section{The strain-vorticity interaction model equation} \label{InteractionModel}

Using the main new identity from \Cref{MainID}, we will prove global regularity for the strain-vorticity interaction model equation. We begin with an a priori estimate for the $\dot{H}^1$ norm.

\begin{proposition} \label{H1Control}
    Suppose $S\in C\left([0,T_{max});H^1_{st}\right)$ is a mild solution of the strain-vorticity interaction model equation \eqref{StrainVortModel}. Then for all $0<t<T_{max}$
    \begin{equation}
\frac{\diff}{\diff t} 
\|S(t)\|_{\dot{H}^1}^2=
-2\left\|-\Delta S\right\|_{L^2}^2,
    \end{equation}
and consequently for all $0<t_1<t_2<T_{max}$,
\begin{equation}
    \frac{1}{2}\|S(\cdot,t_2)\|_{\dot{H}^1}^2+
    \int_{t_1}^{t_2}\|-\Delta S(\cdot,\tau)\|_{L^2}^2 \diff\tau
    =
    \frac{1}{2}\left\|S(\cdot,t_1)\right\|_{\dot{H}^1}^2
\end{equation}
\end{proposition}

\begin{proof}
The result is an almost immediate corollary of Theorem \ref{StrainVortOrthogonality}.
Observe using the higher regularity from Theorem \ref{MildExistence}, that
$S\in C\left((0,T_{max});H^2_{st}\right)$
Computing the derivative directly and applying Proposition \ref{StrainVortOrthogonality}, we find that
\begin{align}
    \frac{\diff}{\diff t}
    \|S(t)\|_{\dot{H}^1}^2
    &=
    -2\|-\Delta S\|_{L^2}^2 +\left<-\Delta S,
    \omega\otimes\omega \right>\\
    &=
    -2\|-\Delta S\|_{L^2}^2.
\end{align}
Integrating this differential equation from $t_1$ to $t_2$ completes the proof.
\end{proof}

We will now prove \Cref{StrainVortGlobalExistIntro}, which is restated for the reader's convenience.

\begin{theorem} \label{StrainVortGlobalExist}
For all $S^0 \in L^2_{st}$, there is a unique, global mild solution of the strain-vorticity interaction model equation,
$S\in C\left([0,+\infty);L^2_{st}\right) \cap 
C\left((0,+\infty);H^\infty\right)$.
Furthermore, if $S^0\in H^1_{st}$,
then for all $0<t<+\infty$
\begin{equation}
    \frac{1}{2}\|S(\cdot,t)\|_{\dot{H}^1}^2+
    \int_0^t\|-\Delta S(\cdot,\tau)\|_{L^2}^2 \diff\tau
    =
    \frac{1}{2}\left\|S^0\right\|_{\dot{H}^1}^2.
\end{equation}
\end{theorem}

\begin{proof}
We know immediately from Theorem \ref{MildExistence}, that there exists a unique mild solution
$S\in C\left([0,T_{max});L^2_{st}\right) \cap 
C\left((0,T_{max});H^\infty\right)$,
so it only remains to show that $T_{max}=+\infty.$
Using the higher regularity from Theorem \ref{MildExistence}, and the control on the $\dot{H}^1$ norm from Proposition \ref{H1Control}, we find that for all, $0<\epsilon<t<T_{max},$
\begin{equation} \label{H1B}
    \|S(\cdot,t)\|_{\dot{H}^1} \leq 
    \|S(\cdot,\epsilon)\|_{\dot{H}^1}.
\end{equation}
The norm $\|S\|_{\dot{H}^1}$ is subcritical with respect to scaling, so if $T_{max}<+\infty$, then
\begin{equation}
    \lim_{t\to T_{max}}\|S(\cdot,t)\|_{\dot{H}^1}=+\infty.
\end{equation}
Therefore the $\dot{H}^1$ control from \eqref{H1B} 
implies that $T_{max}=+\infty.$

Now suppose that additionally $S^0\in H^1_{st}$.
We know from \Cref{H1Control}, that for all $0<\epsilon<t<+\infty$,
\begin{equation}
    \frac{1}{2}\|S(\cdot,t)\|_{\dot{H}^1}^2+
    \int_\epsilon^{t}\|-\Delta S(\cdot,\tau)\|_{L^2}^2 \diff\tau
    =
    \frac{1}{2}\left\|S(\cdot,\epsilon)
    \right\|_{\dot{H}^1}^2.
\end{equation}
Taking the limit $\epsilon\to 0$, we find that
for all $0<t<+\infty$,
\begin{equation}
    \frac{1}{2}\|S(\cdot,t)\|_{\dot{H}^1}^2+
    \int_0^t\|-\Delta S(\cdot,\tau)\|_{L^2}^2 \diff\tau
    =
    \frac{1}{2}\left\|S^0\right\|_{\dot{H}^1}^2,
\end{equation}
and this completes the proof.
\end{proof}

The full Navier--Stokes strain equation can be written in the form,
\begin{equation}
    \partial_t S-\Delta S
    -\frac{1}{2}P_{st}\left(\omega\otimes\omega\right)
    +P_{st}\left((u\cdot\nabla)S+S^2
    +\frac{3}{4}\omega\otimes\omega\right)=0.
\end{equation}
If the term $P_{st}\left((u\cdot\nabla)S+S^2
    +\frac{3}{4}\omega\otimes\omega\right)$ is small enough that the full Navier--Stokes strain equation can be treated as a perturbation of the strain-vorticity interaction model equation, then we will have global regularity. We will now prove \Cref{StrainVortPerturbedRegCritIntro}, which expresses this perturbative condition as a regularity criterion and is restated for the reader's convenience.

\begin{theorem} \label{StrainVortPerturbedRegCrit}
Suppose $u\in C\left([0,T_{max});H^3_{df}\right)$ is a mild solution of the Navier--Stokes equation.
Suppose $0\leq \alpha \leq 1$ and 
$p=\frac{2}{1+\alpha}$.
Then for all $0<t<T_{max}$
\begin{equation} \label{RegCritBound}
    \|S(\cdot,t)\|_{\dot{H}^1}^2
    \leq
    \left\|S^0 \right\|_{\dot{H}^1}^2
    \exp \left( C_\alpha \int_0^t
    \frac{\left\|P_{st}\left((u\cdot\nabla)S +S^2+\frac{3}{4}\omega\otimes\omega\right)
    (\cdot,\tau)\right\|_{\dot{H}^\alpha}^p}
    {\|S(\cdot,\tau)\|_{\dot{H}^1}^p} 
    \diff\tau \right),
\end{equation}
where $C_\alpha$ depends only on $\alpha$.
In particular, if $T_{max}<+\infty,$ then
\begin{equation}
    \int_0^{T_{max}}
    \frac{\left\|P_{st}\left((u\cdot\nabla)S +S^2+\frac{3}{4}\omega\otimes\omega\right)
    (\cdot,t)\right\|_{\dot{H}^\alpha}^p}
    {\|S(\cdot,t)\|_{\dot{H}^1}^p} \diff t 
    =+\infty.
\end{equation}
\end{theorem}

\begin{proof}
First we will observe that the $\dot{H}^1$ norm of $S$ is subcritical with respect to scaling and controls the regularity of $u$, so if $T_{max}<+\infty,$ then
\begin{equation}
    \lim_{t \to T_{max}}
    \|S(\cdot,t)\|_{\dot{H}^1}
    =+\infty.
\end{equation}
Therefore it suffices to prove the bound \eqref{RegCritBound}.

We will use the formulation of the Navier--Stokes equation that treats the equation as a perturbation of the strain-vorticity interaction model equation,
\begin{equation}
    \partial_t S-\Delta S
    -\frac{1}{2}P_{st}\left(\omega\otimes\omega\right)
    +P_{st}\left((u\cdot\nabla)S+S^2
    +\frac{3}{4}\omega\otimes\omega\right)=0.
\end{equation}
For the sake avoiding long expressions let
\begin{equation}
    Q=P_{st}\left((u\cdot\nabla)S+S^2
    +\frac{3}{4}\omega\otimes\omega\right),
\end{equation}
and so we have that
\begin{equation}
    \partial_t S-\Delta S
    -\frac{1}{2}P_{st}\left(\omega\otimes\omega\right)
    +Q=0.
\end{equation}
Applying Theorem \ref{StrainVortOrthogonality}, we know that
\begin{equation}
    \left<-\Delta S, \omega\otimes\omega\right>=0,
\end{equation}
and therefore for all $0<t<T_{max}$,
\begin{equation}
    \frac{\diff}{\diff t} 
    \frac{1}{2} \|S(\cdot,t)\|_{\dot{H}^1}^2
    =
    -\|-\Delta S\|_{L^2}^2-\left<-\Delta S, Q\right>.
\end{equation}

We will first prove the case where $\alpha=0$, then we will prove the case where $\alpha=1,$ and finally we will prove the general case $0<\alpha<1.$
Let $\alpha=0, p=2$.
Applying H\"older's inequality and Young's inequality, both with exponent $2$, we find that
for all $0<t<T_{max}$,
\begin{align}
    \frac{\diff}{\diff t}
    \frac{1}{2} \|S(\cdot,t)\|_{\dot{H}^1}^2
    &\leq
    -\|-\Delta S\|_{L^2}^2
    +\|Q\|_{L^2}\|-\Delta S\|_{L^2} \\
    &\leq
    \frac{1}{4}\|Q\|_{L^2}^2.
\end{align}
Therefore we can see that
for all $0<t<T_{max}$
\begin{equation}
    \frac{\diff}{\diff t}
    \|S(\cdot,t)\|_{\dot{H}^1}^2
    \leq
    \frac{1}{2} \frac{\|Q\|_{L^2}^2}{\|S\|_{\dot{H}^1}^2}
    \|S\|_{\dot{H}^1}^2,
\end{equation}
and so applying Gr\"onwall's inequality we can see that
for all $0<t<T_{max}$
\begin{equation}
    \|S(\cdot,t)\|_{\dot{H}^1}^2
    \leq
    \left\|S^0 \right\|_{\dot{H}^1}^2
    \exp \left( \frac{1}{2} \int_0^t
    \frac{\left\|P_{st}\left((u\cdot\nabla)S +S^2+\frac{3}{4}\omega\otimes\omega\right)
    (\cdot,\tau)\right\|_{L^2}^2}
    {\|S(\cdot,\tau)\|_{\dot{H}^1}^2} 
    \diff\tau \right).
\end{equation}
This completes the proof when $\alpha=0$.

Now let $\alpha=1, p=1$.
Using the duality of $\dot{H}^{-1}$ and $\dot{H}^1$, 
we compute that for all $0<t<T_{max}$,
\begin{align}
    \frac{\diff}{\diff t}
    \frac{1}{2} \|S(\cdot,t)\|_{\dot{H}^1}^2
    &\leq
    -\|-\Delta S\|_{L^2}^2
    + \|-\Delta S\|_{\dot{H}^{-1}}\|Q\|_{\dot{H}^1} \\
    &\leq
    \|Q\|_{\dot{H}^1} \|S\|_{\dot{H}^1}
\end{align}
This implies that
\begin{equation}
    \frac{\diff}{\diff t}
    \|S(\cdot,t)\|_{\dot{H}^1}^2
    \leq
    2 \frac{\|Q\|_{\dot{H}^1}}{\|S\|_{\dot{H}^1}}
    \|S\|_{\dot{H}^1}^2,
\end{equation}
and so applying Gr\"onwall's inequality we may conclude that
\begin{equation}
    \|S(\cdot,t)\|_{\dot{H}^1}^2
    \leq
    \left\|S^0 \right\|_{\dot{H}^1}^2
    \exp \left( 2\int_0^t
    \frac{\left\|P_{st}\left((u\cdot\nabla)S +S^2+\frac{3}{4}\omega\otimes\omega\right)
    (\cdot,\tau)\right\|_{\dot{H}^1}}
    {\|S(\cdot,\tau)\|_{\dot{H}^1}} 
    \diff\tau \right).
\end{equation}
This completes the proof for $\alpha=1$.

Now let $0<\alpha<1,$ and let $p=\frac{2}{1+\alpha}$.
Using the duality of $\dot{H}^\alpha$ and $\dot{H}^{-\alpha},$ and interpolating between $\dot{H}^{-1}$ and $L^2$ we find that
for all $0<t<T_{max}$,
\begin{align}
    \frac{\diff}{\diff t}
    \frac{1}{2} \|S(\cdot,t)\|_{\dot{H}^1}^2
    &\leq
    -\|-\Delta S\|_{L^2}^2
    +\|-\Delta S\|_{\dot{H}^{-\alpha}}
    \|Q\|_{\dot{H}^\alpha} \\
    &\leq
    -\|-\Delta S\|_{L^2}^2
    +\|-\Delta S\|_{\dot{H}^{-1}}^\alpha
    \|-\Delta S\|_{L^2}^{1-\alpha}
    \|Q\|_{\dot{H}^\alpha} \\
    &=
    -\|-\Delta S\|_{L^2}^2
    +\|S\|_{\dot{H}^1}^\alpha
    \|-\Delta S\|_{L^2}^{1-\alpha}
    \|Q\|_{\dot{H}^\alpha}.
\end{align}
Let $q=\frac{2}{1-\alpha}$. Clearly we can see that 
$\frac{1}{p}+\frac{1}{q}=1,$ and that $q(1-\alpha)=2$.
Applying Young's inequality with exponents $p$ and $q$,
we find that
\begin{equation}
    \frac{\diff}{\diff t}
    \frac{1}{2} \|S(\cdot,t)\|_{\dot{H}^1}^2
    \leq \frac{C_{\alpha}}{2}
    \|Q\|_{\dot{H}^\alpha}^p
    \|S\|_{\dot{H}^1}^{\alpha p}.
\end{equation}
Note that $C_\alpha$ depends only on $p$ and $q$, 
and hence only on $\alpha$.
Finally we compute that for all $0<t<T_{max}$,
\begin{align}
    \frac{\diff}{\diff t}
    \|S(\cdot,t)\|_{\dot{H}^1}^2
    &\leq
    C_{\alpha}\frac{\|Q\|_{\dot{H}^\alpha}^p}
    {\|S\|_{\dot{H}^1}^{2-\alpha p}}
    \|S\|_{\dot{H}^1}^2 \\
    &=
    C_{\alpha}\frac{\|Q\|_{\dot{H}^\alpha}^p}
    {\|S\|_{\dot{H}^1}^p}
    \|S\|_{\dot{H}^1}^2,
\end{align}
because 
\begin{align}
    2-\alpha p
    &=
    2\left(1-\frac{\alpha}{1+\alpha}\right) \\
    &= \frac{2}{1+\alpha} \\
    &=p.
\end{align}
Applying Gr\"onwall's inequality we can conclude that 
for all $0<t<T_{max}$,
\begin{equation} 
    \|S(\cdot,t)\|_{\dot{H}^1}^2
    \leq
    \left\|S^0 \right\|_{\dot{H}^1}^2
    \exp \left( C_\alpha \int_0^t
    \frac{\left\|P_{st}\left((u\cdot\nabla)S +S^2+\frac{3}{4}\omega\otimes\omega\right)
    (\cdot,\tau)\right\|_{\dot{H}^\alpha}^p}
    {\|S(\cdot,\tau)\|_{\dot{H}^1}^p} 
    \diff\tau \right),
\end{equation}
and this completes the proof.
\end{proof}

We will now prove \Cref{EndpointRegCritIntro}, which is restated for the reader's convenience.

\begin{theorem} \label{EndpointRegCrit}
Suppose $u\in C\left([0,T_{max});H^3_{df}\right)$ is a mild solution of the Navier--Stokes equation.
Then if $T_{max}<+\infty$, then
\begin{equation}
    \limsup_{t\to T_{max}}
    \frac{\left\|P_{st}\left((u\cdot\nabla)S
    +S^2+\frac{3}{4}\omega\otimes\omega\right)
    (\cdot,t)\right\|_{L^2}}
    {\|-\Delta S(\cdot,t)\|_{L^2}} \geq 1.
\end{equation}
\end{theorem}

\begin{proof}
Suppose towards contradiction that $T_{max}<+\infty,$ and that
\begin{equation}
    \limsup_{t\to T_{max}}
    \frac{\left\|P_{st}\left((u\cdot\nabla)S
    +S^2+\frac{3}{4}\omega\otimes\omega\right)
    (\cdot,t)\right\|_{L^2}}
    {\|-\Delta S(\cdot,t)\|_{L^2}} < 1,
\end{equation}
and again let
\begin{equation}
    Q=P_{st}\left((u\cdot\nabla)S
    +S^2+\frac{3}{4}\omega\otimes\omega\right),
\end{equation}
Then there exists $\epsilon>0,$ such that for all
$T_{max}-\epsilon<t<T_{max},$
\begin{equation}
    \frac{\|Q(\cdot,t)\|_{L^2}}
    {\|-\Delta S(\cdot,t)\|_{L^2}} < 1.
\end{equation}
Applying the same estimate as in Theorem \ref{StrainVortPerturbedRegCrit},
we find that for all $T_{max}-\epsilon<t<T_{max},$
\begin{align}
    \partial_t \|S(\cdot,t)\|_{\dot{H}^1}^2
    &\leq
    -2\|-\Delta S\|_{L^2}^2
    +2\|Q\|_{L^2}\|-\Delta S\|_{L^2} \\
    &< 0.
\end{align}
Therefore we can see that 
for all $T_{max}-\epsilon<t<T_{max}$
\begin{equation}
    \|S(\cdot,t)\|_{\dot{H}^1}^2 <
    \left\|S\left(\cdot,T_{max}-\epsilon\right)
    \right\|_{\dot{H}^1}^2,
\end{equation}
and so
\begin{align}
    \limsup_{t\to T_{max}}\|S(\cdot,t)\|_{\dot{H}^1}^2
    &<
    \left\|S\left(\cdot,T_{max}-\epsilon\right)
    \right\|_{\dot{H}^1}^2 \\
    &<
    +\infty.
\end{align}
This contradicts the assumption that $T_{max}<+\infty$, 
and so this completes the proof.
\end{proof}

\section{Regularity criteria: approximate eigenfunctions of the Laplacian} \label{RegCritEigen}

In this section, we will use the new identity from \Cref{MainID} to develop regularity criteria for the Navier--Stokes equation when the strain matrix is sufficiently close to being an eigenfunction of the Laplacian. We begin by proving \Cref{StrainAlmostEigenIntro}, which is restated for the reader's convenience.

\begin{theorem} \label{StrainAlmostEigen}
Suppose $u\in C\left([0,T_{max});\dot{H}^1_{df}\right)$
is a mild solution to the Navier--Stokes equation,
and suppose $\frac{2}{p}+\frac{3}{q}=2,
\frac{3}{2}<q\leq +\infty$.
Then for all $0<t<T_{max}$
\begin{equation} \label{StrainAlmostEigenBound}
    \|\omega(\cdot,t)\|_{L^2}^2
    <\left\|\omega^0\right\|_{L^2}^2
    \exp\left(C_q
    \int_0^{t}\inf_{\rho\in\mathbb{R}}
    \|-\rho\Delta S-S\|_{L^q}^p \diff\tau \right),
\end{equation}
where $C_q>0$ depends only on $q$.
In particular, if $T_{max}<+\infty$, then
\begin{equation}
    \int_0^{T_{max}}\inf_{\rho\in\mathbb{R}}
    \|-\rho\Delta S-S\|_{L^q}^p \diff t
    =+\infty.
\end{equation}
\end{theorem}

\begin{proof}
We will start by observing that if $T_{max}<+\infty$, then
\begin{equation}
    \lim_{t\to T_{max}}\|\omega(\cdot,t)\|_{L^2}=+\infty,
\end{equation}
so it suffices to prove the bound 
\eqref{StrainAlmostEigenBound}.
We will begin the proof of this bound by recalling the identity for enstrophy growth in terms of vorticity,
\begin{equation}
    \frac{\diff}{\diff t}
    \frac{1}{2}\|\omega(\cdot,t)\|_{L^2}^2
    =-\|\omega\|_{\dot{H}^1}^2
    +\left<S,\omega\otimes\omega\right>.
\end{equation}
We know that for all $0<t<T_{max}, u(\cdot,t)\in H^3_{df}$,
so we can apply Theorem \ref{StrainVortOrthogonality}
to conclude that for all $0<t<T_{max}$,
\begin{equation}
    \left<-\Delta S,\omega\otimes\omega\right>=0.
\end{equation}
Therefore we can see that
for all $0<t<T_{max}$ and for all $\rho\in\mathbb{R}$,
\begin{align}
    \frac{\diff}{\diff t}
    \frac{1}{2}\|\omega(\cdot,t)\|_{L^2}^2
    &=
    -\|\omega\|_{\dot{H}^1}^2
    -\left<-\rho\Delta S-S,\omega\otimes\omega \right> \\
    &\leq
    -\|\omega\|_{\dot{H}^1}^2
    +\|-\rho\Delta S-S\|_{L^q}\|\omega\otimes\omega\|_{L^r} \\
    &=
    -\|\omega\|_{\dot{H}^1}^2
    +\|-\rho\Delta S-S\|_{L^q}\|\omega\|_{L^{2r}}^2,
\end{align}
where $\frac{1}{q}+\frac{1}{r}=1$, and we have applied H\"older's inequality with exponents $q,r$.
Taking the infimum over $\rho$ at each time $0<t<T_{max}$,
we find that for all $0<t<T_{max}$,
\begin{equation}
    \frac{\diff}{\diff t}
    \frac{1}{2}\|\omega(\cdot,t)\|_{L^2}^2
    \leq 
    -\|\omega\|_{\dot{H}^1}^2 +\inf_{\rho\in\mathbb{R}}
    \|-\rho\Delta S-S\|_{L^q}\|\omega\|_{L^{2r}}^2.
\end{equation}
If $q=\infty$, then $r=1$, and so we have
for all $0<t<T_{max}$
\begin{equation}
    \frac{\diff}{\diff t}
    \|\omega(\cdot,t)\|_{L^2}^2
    \leq 
    2\inf_{\rho\in\mathbb{R}}
    \|-\rho\Delta S-S\|_{L^\infty} \|\omega\|_{L^2}^2,
\end{equation}
and applying Gr\"onwall's inequality, we find that
for all $0<t<T_{max}$,
\begin{equation}
    \|\omega(\cdot,t)\|_{L^2}^2
    \leq \left\|\omega^0\right\|_{L^2}^2
    \exp\left(2\int_0^t
    \inf_{\rho\in\mathbb{R}}
    \|-\rho\Delta S-S\|_{L^\infty} 
    \diff\tau \right).
\end{equation}

Now we will consider the case $\frac{3}{2}<q<+\infty$.
In this case we will have $1<r<3$, which in turn implies that
$2<2r<6$. We will let
\begin{equation}
    \lambda=\frac{3}{2r}-\frac{1}{2}.
\end{equation}
Observing that
\begin{equation}
    \lambda\frac{1}{2}+(1-\lambda)\frac{1}{6}=\frac{1}{2r},
\end{equation}
we can interpolate between $L^2$ and $L^6$ and find that
\begin{equation}
    \|\omega\|_{L^{2r}}\leq 
    \|\omega\|_{L^2}^{\frac{3}{2r}-\frac{1}{2}}
    \|\omega\|_{\dot{H}^1}^{\frac{3}{2}-\frac{3}{2r}}.
\end{equation}
Applying this estimate and the Sobolev inequality,
we find that for all $0<t<T_{max}$
\begin{align}
    \frac{\diff}{\diff t}
    \frac{1}{2}\|\omega(\cdot,t)\|_{L^2}^2
    &\leq 
    -\|\omega\|_{\dot{H}^1}^2 +\inf_{\rho\in\mathbb{R}}
    \|-\rho\Delta S-S\|_{L^q}
    \|\omega\|_{L^2}^{\frac{3}{r}-1}
    \|\omega\|_{L^6}^{3-\frac{3}{r}} \\
    &\leq
    -\|\omega\|_{\dot{H}^1}^2 +C \inf_{\rho\in\mathbb{R}}
    \|-\rho\Delta S-S\|_{L^q}
    \|\omega\|_{L^2}^{\frac{3}{r}-1}
    \|\omega\|_{\dot{H}^1}^{3-\frac{3}{r}} \\
    &=
    -\|\omega\|_{\dot{H}^1}^2 +C \inf_{\rho\in\mathbb{R}}
    \|-\rho\Delta S-S\|_{L^q}
    \|\omega\|_{L^2}^{\frac{2}{p}}
    \|\omega\|_{\dot{H}^1}^{\frac{2}{b}},
\end{align}
where $\frac{1}{p}+\frac{1}{b}=1$,
and we have used the fact that
\begin{align}
    \frac{3}{r}-1
    &=
    2-\frac{3}{q} \\
    &=
    \frac{2}{p},
\end{align}
and
\begin{align}
    3-\frac{3}{r}
    &=
    \frac{3}{q} \\
    &=
    2-\frac{2}{p} \\
    &=
    \frac{2}{b}.
\end{align}
Applying Young's inequaltiy with exponents $p,b$ we find that
for all $0<t<T_{max}$,
\begin{equation}
    \partial_t \|\omega(\cdot,t)\|_{L^2}^2
    \leq C_q 
    \inf_{\rho\in\mathbb{R}} 
    \|-\rho\Delta S-S\|_{L^q}^p
    \|\omega\|_{L^2}^2.
\end{equation}
Applying Gr\"onwall's inequality, we find that for all
$0<t<T_{max}$,
\begin{equation}
    \|\omega(\cdot,t)\|_{L^2}^2
    <\left\|\omega^0\right\|_{L^2}^2
    \exp\left(C_q
    \int_0^{t}\inf_{\rho\in\mathbb{R}}
    \|-\rho\Delta S-S\|_{L^q}^p \diff\tau \right),
\end{equation}
and this completes the proof.
\end{proof}

We now compute this infimum explicitly in a number of cases.

\begin{proposition} \label{L2inf}
For all $S\in H^2_{st}$,
\begin{equation}
    \inf_{\rho\in\mathbb{R}}\|-\rho\Delta S-S\|_{L^2}^2
    =\left(1-\frac{\|S\|_{\dot{H}^1}^4}
    {\|S\|_{L^2}^2\|-\Delta S\|_{L^2}^2}\right)
    \|S\|_{L^2}^2
\end{equation}
\end{proposition}

\begin{proof}
Fix $S\in H^2_{st}$, and let
\begin{equation}
    f(\rho)=\|-\rho\Delta S-S\|_{L^2}^2
\end{equation}
Expanding this expression we find that
\begin{equation}
    f(\rho)=\|S\|_{L^2}^2-2\|S\|_{\dot{H}^1}^2 \rho
    +\|-\Delta S\|_{L^2}^2 \rho^2,
\end{equation}
and differentiating we find that
\begin{equation}
    f'(\rho)=-2\|S\|_{\dot{H}^1}^2
    +\|-\Delta S\|_{L^2}^2 \rho.
\end{equation}
It is obvious that this function attains its global minimum at
\begin{equation}
    \rho_0=\frac{\|S\|_{\dot{H}^1}^2}{\|-\Delta S\|_{L^2}^2},
\end{equation}
and therefore we can conclude that
\begin{align}
    \inf_{\rho\in\mathbb{R}}\|-\rho\Delta S-S\|_{L^2}^2
    &=
    \inf_{\rho\in\mathbb{R}} f(\rho) \\
    &=
    f(\rho_0) \\
    &=
    \|S\|_{L^2}^2-
    \frac{\|S\|_{\dot{H}^1}^4}{\|-\Delta S\|_{L^2}^2} \\
    &=
    \left(1-\frac{\|S\|_{\dot{H}^1}^4}
    {\|S\|_{L^2}^2\|-\Delta S\|_{L^2}^2}\right)
    \|S\|_{L^2}^2.
\end{align}
This completes the proof.
\end{proof}

This identity will allow us to express \Cref{StrainAlmostEigen} more explicitly in the case where $q=2$.

\begin{corollary}
Suppose $u\in C\left([0,T_{max});\dot{H}^1_{df}\right)$
is a mild solution to the Navier--Stokes equation.
Then for all $0<t<T_{max}$
\begin{equation} 
    \|\omega(\cdot,t)\|_{L^2}^2
    <\left\|\omega^0\right\|_{L^2}^2
    \exp\left(C_2 \int_0^t
    \left(1-\frac{\|S\|_{\dot{H}^1}^4}
    {\|S\|_{L^2}^2
    \|S\|_{\dot{H}^2}^2}\right)^2
    \|S\|_{L^2}^4 \diff\tau \right),
\end{equation}
where $C_2>0$ is taken as in Theorem \ref{StrainAlmostEigen}
In particular, if $T_{max}<+\infty$, then
\begin{equation}
    \int_0^{T_{max}}
    \left(1-\frac{\|S\|_{\dot{H}^1}^4}
    {\|S\|_{L^2}^2
    \|S\|_{\dot{H}^2}^2}\right)^2
    \|S\|_{L^2}^4 \diff t
    =+\infty.
\end{equation}
\end{corollary}

\begin{proof}
This follows immediately from Theorem \ref{StrainAlmostEigen}
and Proposition \ref{L2inf}.
\end{proof}

\begin{proposition} \label{HilbertInf}
For all $0\leq\alpha <\frac{3}{2},$
and for all $S\in H^{\alpha+2}_{st}$,
\begin{equation}
    \inf_{\rho\in\mathbb{R}}
    \|-\rho\Delta S-S\|_{\dot{H}^\alpha}^2
    =\left(1-\frac{\|S\|_{\dot{H}^{1+\alpha}}^4}
    {\|S\|_{\dot{H}^\alpha}^2
    \|S\|_{\dot{H}^{2+\alpha}}^2}\right)
    \|S\|_{\dot{H}^\alpha}^2
\end{equation}
\end{proposition}

\begin{proof}
Fix $S\in H^{2+\alpha}_{st}$, and let
\begin{equation}
    f(\rho)=\|-\rho\Delta S-S\|_{\dot{H}^\alpha}^2
\end{equation}
Expanding this expression we find that
\begin{equation}
    f(\rho)=\|S\|_{\dot{H}^\alpha}^2
    -2\|S\|_{\dot{H}^{1+\alpha}}^2 \rho
    +\|S\|_{\dot{H}^{2+\alpha}}^2 \rho^2,
\end{equation}
and differentiating we find that
\begin{equation}
    f'(\rho)=-2\|S\|_{\dot{H}^{1+\alpha}}^2
    +\|S\|_{\dot{H}^{2+\alpha}}^2 \rho.
\end{equation}
It is obvious that this function attains its global minimum at
\begin{equation}
    \rho_0=\frac{\|S\|_{\dot{H}^{1+\alpha}}^2}
    {\|S\|_{\dot{H}^{2+\alpha}}^2},
\end{equation}
and therefore we can conclude that
\begin{align}
    \inf_{\rho\in\mathbb{R}}
    \|-\rho\Delta S-S\|_{\dot{H}^\alpha}^2
    &=
    \inf_{\rho\in\mathbb{R}} f(\rho) \\
    &=
    f(\rho_0) \\
    &=
    \|S\|_{\dot{H}^\alpha}^2-
    \frac{\|S\|_{\dot{H}^{1+\alpha}}^4}
    {\|S\|_{\dot{H}^{2+\alpha}}^2} \\
    &=
    \left(1-\frac{\|S\|_{\dot{H}^{1+\alpha}}^4}
    {\|S\|_{\dot{H}^\alpha}^2
    \| S\|_{\dot{H}^{2+\alpha}}^2}\right)
    \|S\|_{\dot{H}^\alpha}^2.
\end{align}
This completes the proof.
\end{proof}

\begin{corollary} \label{StrainAlmostEigenHilbert}
Suppose $u\in C\left([0,T_{max});\dot{H}^1_{df}\right)$
is a mild solution to the Navier--Stokes equation,
and suppose $p=\frac{2}{\alpha+\frac{1}{2}},
0\leq \alpha <\frac{3}{2}$.
Then for all $0<t<T_{max}$
\begin{equation} 
    \|\omega(\cdot,t)\|_{L^2}^2
    <\left\|\omega^0\right\|_{L^2}^2
    \exp\left(C_\alpha \int_0^{t}
    \left(1-\frac{\|S\|_{\dot{H}^{1+\alpha}}^4}
    {\|S\|_{\dot{H}^\alpha}^2
    \| S\|_{\dot{H}^{2+\alpha}}^2}
    \right)^\frac{p}{2}
    \|S\|_{\dot{H}^\alpha}^p
    \diff\tau \right),
\end{equation}
where $C_\alpha>0$ depends only on $\alpha$.
In particular, if $T_{max}<+\infty$, then
\begin{equation}
    \int_0^{T_{max}}
    \left(1-\frac{\|S\|_{\dot{H}^{1+\alpha}}^4}
    {\|S\|_{\dot{H}^\alpha}^2
    \| S\|_{\dot{H}^{2+\alpha}}^2}
    \right)^\frac{p}{2}
    \|S\|_{\dot{H}^\alpha}^p
    \diff t
    =+\infty.
\end{equation}
\end{corollary}

\begin{proof}
We will begin by defining $2\leq q<+\infty$ by
\begin{equation}
    \frac{1}{q}=\frac{1}{2}-\frac{\alpha}{3}.
\end{equation}
Using the fractional Sobolev inequality governing the embedding
$\dot{H}^\alpha\left(\mathbb{R}^3\right)
\hookrightarrow L^q\left(\mathbb{R}^3\right)$,
we can see that for all $\rho\in \mathbb{R}$,
\begin{equation}
    \|-\rho\Delta S-S\|_{L^q}
    \leq 
    \Tilde{C}_\alpha
    \|-\rho\Delta S-S\|_{\dot{H}^\alpha}.
\end{equation}
Taking the infimum over $\rho\in\mathbb{R}$
and applying Proposition \ref{HilbertInf}, we find that
\begin{align}
    \inf_{\rho\in\mathbb{R}}
    \|-\rho\Delta S-S\|_{L^q}
    &\leq 
    \Tilde{C}_\alpha
    \inf_{\rho\in\mathbb{R}}
    \|-\rho\Delta S-S\|_{\dot{H}^\alpha} \\
    &=
    \Tilde{C}_\alpha
    \left(1-\frac{\|S\|_{\dot{H}^{1+\alpha}}^4}
    {\|S\|_{\dot{H}^\alpha}^2
    \| S\|_{\dot{H}^{2+\alpha}}^2}
    \right)^\frac{1}{2}
    \|S\|_{\dot{H}^\alpha}.
\end{align}
Finally we observe that
\begin{equation}
    \frac{2}{p}+\frac{3}{q}=2,
\end{equation}
and for all $0<t<T_{max},$
\begin{equation}
    \inf_{\rho\in\mathbb{R}}
    \|-\rho\Delta S-S\|_{L^q}^p
    \leq
    C_\alpha
    \left(1-\frac{\|S\|_{\dot{H}^{1+\alpha}}^4}
    {\|S\|_{\dot{H}^\alpha}^2
    \| S\|_{\dot{H}^{2+\alpha}}^2}
    \right)^\frac{p}{2}
    \|S\|_{\dot{H}^\alpha}^p.
\end{equation}
Applying Theorem \ref{StrainAlmostEigen},
this completes the proof.
\end{proof}

We conclude this paper by considering the endpoint case, proving \Cref{StrainAlmostEigenEndpointIntro},
which is restated for the reader's convenience.

\begin{theorem} \label{StrainAlmostEigenEndpoint}
Suppose $u\in C\left([0,T_{max});\dot{H}^1_{df}\right)$
is a mild solution to the Navier--Stokes equation
that blows up in finite-time $T_{max}<+\infty$.
Then
\begin{equation}
    \limsup_{t\to T_{max}} 
    \inf_{\rho\in\mathbb{R}}
    \|-\rho\Delta S-S\|_{L^\frac{3}{2}}
    \geq
    \frac{1}{3}
    \left(\frac{2}{\pi}\right)^\frac{4}{3}.
\end{equation}
\end{theorem}

\begin{proof}
Proceeding as in the proof of Theorem \ref{StrainAlmostEigen},
We can see that for all $0<t<T_{max}$,
and for all $\rho\in\mathbb{R}$
\begin{align}
    \frac{\diff}{\diff t}
    \frac{1}{2}\|\omega(\cdot,t)\|_{L^2}^2
    &=
    -\|\omega\|_{\dot{H}^1}^2
    +\left<S,\omega\otimes\omega\right> \\
    &=
    -\|\omega\|_{\dot{H}^1}^2
    -\left<-\rho\Delta S-S,\omega\otimes\omega\right>.
\end{align}
Applying H\"older's inequality with exponents $\frac{3}{2},6,6$,
applying the Sobolev inequality, 
and taking the infimum over $\rho\in\mathbb{R}$,
we find that for all $0<t<T_{max}$,
\begin{align}
    \frac{\diff}{\diff t}
    \frac{1}{2}\|\omega(\cdot,t)\|_{L^2}^2
    &\leq
    -\|\omega\|_{\dot{H}^1}^2
    +\inf_{\rho\in\mathbb{R}}
    \|-\rho\Delta S-S\|_{L^\frac{3}{2}}
    \|\omega\|_{L^6}^2 \\
    &\leq
    -\|\omega\|_{\dot{H}^1}^2
    +3\left(\frac{\pi}{2}\right)^\frac{4}{3}
    \inf_{\rho\in\mathbb{R}}
    \|-\rho\Delta S-S\|_{L^\frac{3}{2}}
    \|\omega\|_{\dot{H}^1}^2 \\
    &=
    -\|\omega\|_{\dot{H}^1}^2
    \left(1-3\left(\frac{\pi}{2}\right)^\frac{4}{3}
    \inf_{\rho\in\mathbb{R}}
    \|-\rho\Delta S-S\|_{L^\frac{3}{2}}\right).
\end{align}

Suppose towards contradiction that 
\begin{equation}
    \limsup_{t\to T_{max}} 
    \inf_{\rho\in\mathbb{R}}
    \|-\rho\Delta S-S\|_{L^\frac{3}{2}}
    <
    \frac{1}{3}
    \left(\frac{2}{\pi}\right)^\frac{4}{3}.
\end{equation}
Then clearly there exists $\epsilon>0$ such that,
for all $T_{max}-\epsilon<t<T_{max},$
\begin{equation}
3\left(\frac{\pi}{2}\right)^\frac{4}{3}
    \inf_{\rho\in\mathbb{R}}
    \|-\rho\Delta S-S\|_{L^\frac{3}{2}}
    <
    1,
\end{equation}
and consequently
\begin{equation}
    \frac{\diff}{\diff t}
    \|\omega(\cdot,t)\|_{L^2}^2<0.
\end{equation}
This implies that for all $T_{max}-\epsilon<t<T_{max}$,
\begin{equation}
    \|\omega(\cdot,t)\|_{L^2}^2
    <
    \|\omega(\cdot,T_{max}-\epsilon)\|_{L^2},
\end{equation}
and consequently that
\begin{equation}
    \limsup_{t\to T_{max}}\|\omega\|_{L^2}^2
    <
    \|\omega(\cdot,T_{max}-\epsilon)\|_{L^2}<+\infty.
\end{equation}
This contradicts our assumption that $T_{max}<+\infty$,
and so this completes the proof.
\end{proof}

\section*{Acknowledgments}
This publication was supported in part by the Fields Institute for Research in the Mathematical Sciences while the author was in residence during the Fall 2020 semester.
The author would like to thank the anonymous referee for a very thorough reading of the manuscript and a number of helpful suggestions which have significantly improved the clarity of this paper.

\bibliographystyle{plain}
\bibliography{Bib}

\begin{bibdiv}
\begin{biblist}

\bib{CarboneBragg}{article}{
      author={Carbone, M.},
      author={Bragg, A.~D.},
       title={Is vortex stretching the main cause of the turbulent energy
  cascade?},
        date={2020},
        ISSN={0022-1120},
     journal={J. Fluid Mech.},
      volume={883},
       pages={R2, 13},
         url={https://doi.org/10.1017/jfm.2019.923},
      review={\MR{4041093}},
}

\bib{ConstantinLaxMajda}{article}{
      author={Constantin, P.},
      author={Lax, P.~D.},
      author={Majda, A.},
       title={A simple one-dimensional model for the three-dimensional
  vorticity equation},
        date={1985},
        ISSN={0010-3640},
     journal={Comm. Pure Appl. Math.},
      volume={38},
      number={6},
       pages={715\ndash 724},
         url={https://doi.org/10.1002/cpa.3160380605},
      review={\MR{812343}},
}

\bib{DeGregorio}{article}{
      author={De~Gregorio, Salvatore},
       title={On a one-dimensional model for the three-dimensional vorticity
  equation},
        date={1990},
        ISSN={0022-4715},
     journal={J. Statist. Phys.},
      volume={59},
      number={5-6},
       pages={1251\ndash 1263},
         url={https://doi.org/10.1007/BF01334750},
      review={\MR{1063199}},
}

\bib{EGM}{article}{
      author={Elgindi, Tarek~M.},
      author={Ghoul, Tej-Eddine},
      author={Masmoudi, Nader},
       title={Stable self-similar blow-up for a family of nonlocal transport
  equations},
        date={2021},
        ISSN={2157-5045,1948-206X},
     journal={Anal. PDE},
      volume={14},
      number={3},
       pages={891\ndash 908},
         url={https://doi.org/10.2140/apde.2021.14.891},
      review={\MR{4259877}},
}

\bib{ElgindiJeong}{article}{
      author={Elgindi, Tarek~M.},
      author={Jeong, In-Jee},
       title={On the effects of advection and vortex stretching},
        date={2020},
        ISSN={0003-9527,1432-0673},
     journal={Arch. Ration. Mech. Anal.},
      volume={235},
      number={3},
       pages={1763\ndash 1817},
         url={https://doi.org/10.1007/s00205-019-01455-9},
      review={\MR{4065651}},
}

\bib{KatoFujita}{article}{
      author={Fujita, Hiroshi},
      author={Kato, Tosio},
       title={On the {N}avier-{S}tokes initial value problem. {I}},
        date={1964},
        ISSN={0003-9527},
     journal={Arch. Rational Mech. Anal.},
      volume={16},
       pages={269\ndash 315},
         url={http://dx.doi.org/10.1007/BF00276188},
      review={\MR{0166499}},
}

\bib{Hou2018}{article}{
      author={Hou, Thomas~Y.},
      author={Jin, Tianling},
      author={Liu, Pengfei},
       title={Potential singularity for a family of models of the axisymmetric
  incompressible flow},
        date={2018},
        ISSN={0938-8974},
     journal={J. Nonlinear Sci.},
      volume={28},
      number={6},
       pages={2217\ndash 2247},
         url={https://doi.org/10.1007/s00332-017-9370-9},
      review={\MR{3867642}},
}

\bib{HouLei}{article}{
      author={Hou, Thomas~Y.},
      author={Lei, Zhen},
       title={On the stabilizing effect of convection in three-dimensional
  incompressible flows},
        date={2009},
        ISSN={0010-3640},
     journal={Comm. Pure Appl. Math.},
      volume={62},
      number={4},
       pages={501\ndash 564},
         url={https://doi.org/10.1002/cpa.20254},
      review={\MR{2492706}},
}

\bib{HouEtAl}{article}{
      author={Hou, Thomas~Y.},
      author={Lei, Zhen},
      author={Luo, Guo},
      author={Wang, Shu},
      author={Zou, Chen},
       title={On finite time singularity and global regularity of an
  axisymmetric model for the 3{D} {E}uler equations},
        date={2014},
        ISSN={0003-9527},
     journal={Arch. Ration. Mech. Anal.},
      volume={212},
      number={2},
       pages={683\ndash 706},
         url={https://doi.org/10.1007/s00205-013-0717-6},
      review={\MR{3176355}},
}

\bib{HQWW}{article}{
      author={Huang, De},
      author={Qin, Xiang},
      author={Wang, Xiuyuan},
      author={Wei, Dongyi},
       title={Self-similar finite-time blowups with smooth profiles of the
  generalized {C}onstantin-{L}ax-{M}ajda model},
        date={2024},
        ISSN={0003-9527,1432-0673},
     journal={Arch. Ration. Mech. Anal.},
      volume={248},
      number={2},
       pages={Paper No. 22, 65},
         url={https://doi.org/10.1007/s00205-024-01971-3},
      review={\MR{4721709}},
}

\bib{JiaStewartSverak}{article}{
      author={Jia, H.},
      author={Stewart, S.},
      author={Sverak, V.},
       title={On the {D}e {G}regorio modification of the
  {C}onstantin-{L}ax-{M}ajda model},
        date={2019},
        ISSN={0003-9527},
     journal={Arch. Ration. Mech. Anal.},
      volume={231},
      number={2},
       pages={1269\ndash 1304},
         url={https://doi.org/10.1007/s00205-018-1298-1},
      review={\MR{3900823}},
}

\bib{Leray}{article}{
      author={Leray, Jean},
       title={Sur le mouvement d'un liquide visqueux emplissant l'espace},
        date={1934},
        ISSN={0001-5962},
     journal={Acta Math.},
      volume={63},
      number={1},
       pages={193\ndash 248},
         url={http://dx.doi.org/10.1007/BF02547354},
      review={\MR{1555394}},
}

\bib{Lieb}{article}{
      author={Lieb, Elliott~H.},
       title={Sharp constants in the {H}ardy-{L}ittlewood-{S}obolev and related
  inequalities},
        date={1983},
        ISSN={0003-486X},
     journal={Ann. of Math. (2)},
      volume={118},
      number={2},
       pages={349\ndash 374},
         url={https://doi.org/10.2307/2007032},
      review={\MR{717827}},
}

\bib{MillerAlmost2D}{article}{
      author={Miller, Evan},
       title={Global regularity for solutions of the three dimensional
  {N}avier-{S}tokes equation with almost two dimensional initial data},
        date={2020},
        ISSN={0951-7715,1361-6544},
     journal={Nonlinearity},
      volume={33},
      number={10},
       pages={5272\ndash 5323},
         url={https://doi.org/10.1088/1361-6544/ab9246},
      review={\MR{4143973}},
}

\bib{MillerStrain}{article}{
      author={Miller, Evan},
       title={A regularity criterion for the {N}avier-{S}tokes equation
  involving only the middle eigenvalue of the strain tensor},
        date={2020},
        ISSN={0003-9527},
     journal={Arch. Ration. Mech. Anal.},
      volume={235},
      number={1},
       pages={99\ndash 139},
         url={https://doi.org/10.1007/s00205-019-01419-z},
      review={\MR{4062474}},
}

\bib{MillerEigen}{article}{
      author={Miller, Evan},
       title={Global regularity for solutions of the {N}avier-{S}tokes equation
  sufficiently close to being eigenfunctions of the {L}aplacian},
        date={2021},
        ISSN={2330-1511},
     journal={Proc. Amer. Math. Soc. Ser. B},
      volume={8},
       pages={129\ndash 144},
         url={https://doi.org/10.1090/bproc/62},
      review={\MR{4273161}},
}

\bib{MillerStrainModel}{article}{
      author={Miller, Evan},
       title={Finite-time blowup for a {N}avier-{S}tokes model equation for the
  self-amplification of strain},
        date={2023},
        ISSN={2157-5045,1948-206X},
     journal={Anal. PDE},
      volume={16},
      number={4},
       pages={997\ndash 1032},
         url={https://doi.org/10.2140/apde.2023.16.997},
      review={\MR{4605202}},
}

\bib{MillerSawyer}{article}{
      author={Miller, Evan},
      author={Sawyer, Eric},
       title={A {H}elmholtz-type decomposition for the space of symmetric
  matrices},
        date={2023},
        ISSN={2330-0000},
     journal={Trans. Amer. Math. Soc. Ser. B},
      volume={10},
       pages={1449\ndash 1493},
         url={https://doi.org/10.1090/btran/167},
      review={\MR{4672124}},
}

\bib{NeuPen1}{incollection}{
      author={Neustupa, Ji\v{r}\'{\i}},
      author={Penel, Patrick},
       title={Anisotropic and geometric criteria for interior regularity of
  weak solutions to the 3{D} {N}avier-{S}tokes equations},
        date={2001},
   booktitle={Mathematical fluid mechanics, recent results and open questions},
      editor={Neustupa, J.},
      editor={Penel, P.},
      series={Adv. Math. Fluid Mech.},
   publisher={Birkh\"{a}user, Basel},
       pages={237\ndash 265},
      review={\MR{1865056}},
}

\bib{NeuPen2}{article}{
      author={Neustupa, Ji\v{r}\'{\i}},
      author={Penel, Patrick},
       title={The role of eigenvalues and eigenvectors of the symmetrized
  gradient of velocity in the theory of the {N}avier-{S}tokes equations},
        date={2003},
        ISSN={1631-073X},
     journal={C. R. Math. Acad. Sci. Paris},
      volume={336},
      number={10},
       pages={805\ndash 810},
         url={https://doi.org/10.1016/S1631-073X(03)00174-2},
      review={\MR{1990019}},
}

\bib{Sobolev}{book}{
      author={Sobolev, S.~L.},
       title={Some applications of functional analysis in mathematical
  physics},
      series={Translations of Mathematical Monographs},
   publisher={American Mathematical Society, Providence, RI},
        date={1991},
      volume={90},
        ISBN={0-8218-4549-7},
        note={Translated from the third Russian edition by Harold H. McFaden,
  With comments by V. P. Palamodov},
      review={\MR{1125990}},
}

\bib{Talenti}{article}{
      author={Talenti, Giorgio},
       title={Best constant in {S}obolev inequality},
        date={1976},
        ISSN={0003-4622},
     journal={Ann. Mat. Pura Appl. (4)},
      volume={110},
       pages={353\ndash 372},
         url={https://doi.org/10.1007/BF02418013},
      review={\MR{0463908}},
}

\end{biblist}
\end{bibdiv}

\end{document}